\documentclass[12pt]{amsart}

\usepackage[all]{xy}
\usepackage{amscd,amsmath,amsthm,amssymb}
\usepackage{mathtools}
\usepackage{mathrsfs}
\usepackage{tikz-cd}

\definecolor{cadmiumgreen}{rgb}{0.0, 0.42, 0.24}
\usepackage[
colorlinks, citecolor=cadmiumgreen,
pdfauthor={Soham Ghosh, Farbod Shokrieh}, 
pdfstartview ={FitV},
]{hyperref}
\hypersetup{
pdftitle={Tropical Poincar\'e bundle, Fourier--Mukai transform, and a generalized Poincar\'e formula},
}

\usepackage[
alphabetic,
msc-links,
nobysame,
lite,
]{amsrefs} 

\usepackage[normalem]{ulem}

\usepackage[left=3.5cm,top=3.5cm,right=3.5cm]{geometry}
\usepackage{microtype}
\mathtoolsset{showonlyrefs}

\newtheorem{thm}{Theorem}[section]
\newtheorem{theorem}[thm]{Theorem}
\newtheorem*{theorem*}{Theorem}
\newtheorem{proposition}[thm]{Proposition}
\newtheorem*{proposition*}{Proposition}
\newtheorem{corollary}[thm]{Corollary}
\newtheorem{lemma}[thm]{Lemma}

\theoremstyle{remark}
\newtheorem{remark}[thm]{Remark}

\newtheorem{notation}[equation]{Notation}

\theoremstyle{definition}
\newtheorem{definition}[thm]{Definition}

\DeclareMathOperator{\End}{\ensuremath{End}}
\DeclareMathOperator{\Hom}{\ensuremath{Hom}}

\DeclareMathOperator{\rank}{\ensuremath{rank}}

\DeclareMathOperator{\Jac}{\ensuremath{Jac}}

\DeclareMathOperator{\Aff}{\ensuremath{Aff}}
\DeclareMathOperator{\pic}{\ensuremath{Pic}}

\newcommand{\dual}[1]{\widehat{#1}}

\newcommand{\bR}{\mathbb{R}}

\newcommand{\bZ}{\mathbb{Z}}

\newcommand{\sL}{\mathscr{L}}

\numberwithin{equation}{section}

\subjclass[2020]{
\href{https://mathscinet.ams.org/msc/msc2020.html?t=14T20}{14T20},
\href{https://mathscinet.ams.org/msc/msc2020.html?t=11G10}{11G10},
\href{https://mathscinet.ams.org/msc/msc2020.html?t=14C25}{14C25},
\href{https://mathscinet.ams.org/msc/msc2020.html?t=14C17}{14C17},
\href{https://mathscinet.ams.org/msc/msc2020.html?t=14H40}{14H40},
\href{https://mathscinet.ams.org/msc/msc2020.html?t=14C40}{14C40}
}

\begin{document}

\title[Tropical Poincar\'e bundle, Fourier--Mukai, and Poincar\'e formula]{Tropical Poincar\'e bundle, Fourier--Mukai transform, and a generalized Poincar\'e formula}

\author{Soham Ghosh}
\address{University of Washington \\ Box 354350 \\ Seattle, WA 98195 \\ USA }
\email{\href{mailto:soham13@uw.edu}{soham13@uw.edu}}

\author{Farbod Shokrieh}
\address{University of Washington \\ Box 354350 \\ Seattle, WA 98195 \\ USA }
\email{\href{mailto:farbod@uw.edu}{farbod@uw.edu}}

\date{\today}

\begin{abstract} 
We construct a tropical analogue of the Poincar\'e bundle and prove a (cohomological) Fourier--Mukai transform for real tori with integral structures. We then prove a tropical analogue of  Beauville's generalized Poincar\'e formula for polarized abelian varieties. Some consequences include a geometric Riemann--Roch theorem for tropical abelian varieties, as well as a tropical Poincar\'e--Prym formula which was recently conjectured by R\"ohrle and Zakharov.
\end{abstract}

\maketitle

\setcounter{tocdepth}{1}
\tableofcontents

\thispagestyle{empty}

\section{Introduction} \label{sec:intro}
\renewcommand*{\thethm}{\Alph{thm}}

\subsection{Background}
Let $A$ be a complex abelian variety of dimension $g$ and let $\widehat{A}\coloneqq\pic^0(A)$ be its dual. In general, $A$ and $\widehat{A}$ are not isomorphic as projective varieties. Nevertheless, there exists a fundamental duality theory relating $A$ and $\widehat{A}$ via the Fourier--Mukai transform. Classically, this is a well-understood theory for abelian varieties, first studied at the level of cohomology by Lieberman (\cite[Appendix 2]{Kleiman}), and put into the framework of derived equivalences in the seminal work of Mukai (\cite{Mukai}).

Beauville, in \cite{Beauv2, Beauv1}, considers the Fourier--Mukai theory at the level of Chow rings and studies the endomorphisms on  Chow ring induced by the $n$-multiplication maps on $A$, leading to the Beauville decomposition of Chow groups. Furthermore, Beauville computes the Fourier--Mukai transform of intersection powers of symmetric ample line bundles on abelian varieties as follows (\cite[Proposition 5]{Beauv2} -- see also \cite{CAV}*{Theorem 16.5.5}). 
\begin{theorem*}[Beauville]\label{thm:classical}
    Let $L$ be a symmetric ample line bundle on $A$ with $d=h^0(L)$ and let $\phi_L:A\rightarrow \widehat{A}$ be the corresponding polarization map. Let $F: \operatorname{Ch}^\bullet(A)_{\mathbb{Q}}\rightarrow \operatorname{Ch}^\bullet(\widehat{A})_{\mathbb{Q}}$ be the Fourier--Mukai transform on the rational Chow ring. Then the following equality holds in $\operatorname{Ch}^\bullet(\widehat{A})_{\mathbb{Q}}$, for any $0\leq p\leq g$: 
        \[ F\left(\frac{L^{\cdot p}}{p!}\right)=\frac{(-1)^{g-p}}{d}\phi_{L\ast}\left(\frac{L^{\cdot g-p}}{(g-p)!}\right) \, .\]
\end{theorem*}

The classical Poincar\'e formula relates tautological cycles and the Riemann's theta divisor on the Jacobian of a curve (see, for example, \cite[p.350]{GHPAG}, \cite[\S11.2]{CAV}). Using the previous theorem, Beauville proves a beautiful generalization as follows (\cite[Corollaire 2]{Beauv2} -- see also \cite{CAV}*{\S16.5}). Let $\star$ be the Pontryagin product on $\operatorname{Ch}^\bullet(A)$.
\begin{theorem*}[Beauville]\label{thm:genPoin}
Let $L$ be a symmetric ample line bundle on $A$ with $d=h^0(L)$ and let $c_L\coloneqq{L^{\cdot g-1}}/{(d(g-1)!)}\in \operatorname{Ch}_1(A)_\mathbb{Q}$. Then the following equality holds in $\operatorname{Ch}^p(A)_\mathbb{Q}$, for all $0\leq p\leq g$:
\[\frac{L^{\cdot p}}{p!}=d\frac{c_L^{\star g-p}}{(g-p)!}\, .\]
\end{theorem*}

 \subsection{Our contribution}
  Our first goal is to define a theory of Fourier--Mukai transform for tropical abelian varieties at the level of tropical cohomology. We do this, more generally, for real tori with integral structures. In particular, we define and construct the tropical Poincar\'e bundle $\mathcal{P}_X$ on the product of a real torus with integral structure $X$ and its dual $\widehat{X}$, using the tropical Appell--Humbert theorem (see Theorem~\ref{Poincare:existence}). 
  Let $p_1:X\times\widehat{X}\rightarrow X$ and $p_2:X\times\widehat{X}\rightarrow \widehat{X}$ be the projection maps. We can then formally define a cohomological tropical Fourier--Mukai transform (see Definitions~\ref{Def:cohomtropFMker} and \ref{def:FMtrans}) as follows:
 \begin{equation}
    F=F_{E_X}: H^{\bullet, \bullet}(X)_{\mathbb{Q}}\rightarrow H^{\bullet, \bullet}(\widehat{X})_{\mathbb{Q}}\, ,  \ \ F_{E_X}(\alpha)=p_{2\ast}(e^{c_1(\mathcal{P}_X)}\cdot p_1^{\ast}\alpha) \, .
\end{equation}
We also show that the above formal definition is closely related to the tropical Poincar\'e duality map (see Proposition~\ref{prop: fourierdesc}). Consequently, we obtain our first main result.

\begin{theorem}[=Corollary~\ref{cor:FMpq}]\label{ThmMain1}
$F_{E_X}|_{H^{p,q}(X)}: H^{p,q}(X)\xrightarrow{\sim} H^{g-q,g-p}(\widehat{X})$ is an isomorphism of abelian groups.
\end{theorem}

Our second goal is to apply the tropical Fourier--Mukai theory to obtain tropical analogues of the aforementioned classical results of Beauville.

\begin{theorem}[=Theorem~\ref{mainthm}]\label{ThmMain2}
    Let $\sL$ be a nondegenerate line bundle on a real torus  with integral structure $X=N_\bR/\Lambda$ of dimension $g$. Let the first Chern class $c_1(\sL)$ of $\sL$ be $E\in H^{1,1}(X)=(\Lambda\otimes N)^{\ast}$. Let $F_{E_X}:H^{\bullet, \bullet}(X)\rightarrow H^{\bullet, \bullet}(\widehat{X})$ be the Fourier--Mukai transform on tropical cohomology. Then the following equality holds in $H^{g-p, g-p}(\widehat{X})$, for any $0\leq p\leq g$:
    \[ F_{E_X}\left(\frac{c_1(\sL^{\cdot p})}{p!}\right)=\frac{(-1)^{g-p}}{\det E}\phi_{\sL\ast}\left(\frac{c_1(\sL^{\cdot g-p})}{(g-p)!}\right) \, .\]
\end{theorem}

We prove Theorem~\ref{ThmMain2} by giving explicit formulas for both sides of the identity. It is intriguing that our computations of the above intersection theoretic identity  involve concrete and classical linear algebraic identities such as a generalized Cauchy--Binet formula and Jacobi's identity for determinants.

In the case that $\sL$ is ample, we verify that $\det E$ in Theorem~\ref{ThmMain2} has a cohomological as well as geometric interpretation. For this, we use a tropical analogue of the analytic Riemann--Roch for abelian varieties (see Lemma~\ref{rrlemma} and Corollary~\ref{lem:tropanRR}). 
Consequently, we obtain the following generalized tropical Poincar\'e formula.

\begin{theorem}[=Proposition~\ref{prop:tropgenpoincare}]\label{ThmMain3}
    Let $[D]\in \operatorname{CaCl}(X)$ be an ample tropical Cartier divisor class on $X=N_\bR/\Lambda$ with $d=h^0(X,\sL(D))$. Let $c_{[D]}=[D]^{\cdot g-1}/(d(g-1)!)$, where $[D]^{\cdot g-1}$ is the $(g-1)$-fold tropical intersection product of $D$. Then, for $0\leq p\leq g$, we have the equality:
\begin{equation}
\frac{[D]^{\cdot p}}{p!}=d\frac{c_{[D]}^{\star \ g-p}}{(g-p)!}\, .\end{equation}
in $H_{p,p}(X)$. Here $\star$ denotes the Pontryagin product on tropical homology.
\end{theorem}

As an application of Theorem~\ref{ThmMain3}, we obtain a tropical analogue of {\em geometric Riemann--Roch} for tropical abelian varieties (see Remark~\ref{rem:gRR} for a comparison with \cite[Theorem~47]{KS}).

\begin{theorem}[=Corollary~\ref{cor:gRR}]\label{ThmMain4}
    Let $\mathscr{L}\coloneqq \mathscr{L}(D)$ be an ample line bundle on a tropical abelian variety $X$ of dimension $g$ with $d=h^0(X, \mathscr{L})$. Then:
    \[d=\frac{(D^{\cdot g})}{g!}\coloneqq\int_X\frac{c_1(\mathscr{L})^{\wedge g}}{g!}\, ,\]
    where $(D^g)\coloneqq \int_X c_1(\mathscr{L})^{\wedge g}$ is the $g$-fold intersection number of $D$, given by integration of tropical Dolbeault superforms.
\end{theorem}

For tropical Jacobians, one recovers tropical Poincar\'e formula in \cite[Theorem A]{GSJ} about the image $\widetilde{W}_d$ of the tropical Abel--Jacobi maps $\Gamma^d\rightarrow\Jac(\Gamma^d)$ from Theorem~\ref{ThmMain3} (see \S\ref{subsubsectropJac}). 

One can also apply Theorem~\ref{ThmMain3} to the setting of (continuous) tropical Prym variety $\operatorname{Prym}_c(\widetilde{\Gamma}/\Gamma)$ defined by R\"ohrle and Zakharov in \cite{RZ}. 
Let $\zeta_c$ be the canonical principal polarization for $\operatorname{Prym}_c(\widetilde{\Gamma}/\Gamma)$ defined in \cite[\S4.4]{RZ}. We obtain the following tropical Poincar\'e--Prym formula. This resolves a conjecture by R\"ohrle and Zakharov (\cite[Conjecture~4.24]{RZ}). 

\begin{theorem}[=Corollary~\ref{conPP}]\label{ThmMain5}
    Let $\widetilde{\Gamma} \rightarrow \Gamma$ be a double cover of tropical curves, let $q \in \widetilde{\Gamma}$ be a basepoint and let $g_0=\operatorname{dim} \operatorname{Prym}_c(\widetilde{\Gamma} / \Gamma)\coloneqq g(\widetilde{\Gamma})-g(\Gamma)$. Then, for all $1\leq d\leq g_0$, the following equality holds in $H_{d, d}(\operatorname{Prym}(\widetilde{\Gamma} / \Gamma))$:
$$
\frac{\Psi^d_{q,*}\operatorname{cyc}[\widetilde{\Gamma^d}]}{d!}=\frac{2^{d}}{\left(g_0-d\right) !}[\zeta_c]^{g_0-d} \, ,
$$
 where $[\zeta_c]\in H_{g_0-1,g_0-1}(\operatorname{Prym}(\widetilde{\Gamma}/\Gamma))$ is the class of the principal polarization of $\operatorname{Prym}_c(\widetilde{\Gamma} / \Gamma)$ and $\Psi^d_{q,*}:\widetilde{\Gamma^d}\rightarrow \operatorname{Prym}_c(\widetilde{\Gamma}/\Gamma)$ is the $d$-fold tropical Abel--Prym map with respect to basepoint $q\in\widetilde{\Gamma}$.
\end{theorem}

A real torus with integral structure may be thought of as a canonical skeleton of a non-archimedean Berkovich analytic abelian variety (see \cite[\S6.5]{BerkovichBook}, \cite{GublerAV}). As such, one may consider our results as non-archimedean counterparts of existing classical results.  It is also natural to expect that our tropical identities control the limiting behavior in degenerating families of abelian varieties. This expectation is compatible with the setup in \cite[\S6]{Mum1}, \cite{AN} and \cite[Chapter 6, \S1]{FCAV}. We remark that, in this direction, an explicit formula for limit of the Fourier--Mukai transform, in the special case of rank one degenerations of abelian varieties, is obtained in \cite{vdGK}. 

\subsection{Structure of the paper}

In \S\ref{sec:GenPrelim} we review some general background about tropical spaces. 
In \S\ref{sec:RealtoriPrelim} we provide a brief account of the theory of real tori with integral structures and line bundles on them. We also record some consequences of the tropical Appell--Humbert theorem, which do not have an explicit account in the literature. For example, we prove a tropical theorem of the square and we give the canonical homomorphism from a real torus with integral structure to its dual associated to a line bundle. 
In \S\ref{sec:Poincarebundle}, we define and construct the Poincar\'e line bundle and, along the way, we prove a weak form of the seesaw principle for real tori with integral structures. 
In \S\ref{sec:tropFM} we define and describe the Fourier--Mukai transform on tropical cohomology, explicitly comparing it with the Poincar\'e duality isomorphism, using which we prove Theorem~\ref{ThmMain1}. 
In \S\ref{sec:GenPoincareformula} we study the Fourier--Mukai action on intersection powers of line bundles to prove Theorem~\ref{ThmMain2} and Theorem~\ref{ThmMain3}. We also briefly discuss a tropical analogue of analytic Riemann--Roch for abelian varieties. 
In \S\ref{sec:applications} we provide some applications of the results of the previous section. In particular, we prove Theorem~\ref{ThmMain4} and provide applications to tropical Jacobians and Prym varieties, ending with a proof of Theorem~\ref{ThmMain5}.

\subsection*{Acknowledgments}
We would like to thank Andreas Gross, Junaid Hasan, Klaus K\"unnemann and Caelan Ritter for helpful discussions. The work was partially supported by NSF CAREER DMS-2044564 grant.

\renewcommand*{\thethm}{\arabic{section}.\arabic{thm}}

\subsection*{Notations and conventions}
We list some of the notations and conventions that will be used throughout the paper.
\begin{itemize}
    \item[(i)] For a free abelian group $M$, we denote the real vector space $M\otimes_\mathbb{Z}\mathbb{R}$ by $M_\bR$. We denote its dual $\Hom(M, \mathbb{Z})$ by $M^\ast$. 
    \item[(ii)] For a basis $\{m_1,\dots, m_n\}$ of a free abelian group $M$, we will consider the dual basis $\{m_1^*,\dots, m_n^*\}$ for $M^\ast$  defined by the normalized evaluation pairing $m_i^*(m_i)=1$ and $m_i^*(m_j)=0$ for all $1\leq i\neq j\leq n$. 
    \item[(iii)] For any integer $n>0$, we denote the symmetric group on $n$ letters by $\mathfrak{S}_n$.
\end{itemize}

\section{Preliminaries on tropical spaces}\label{sec:GenPrelim}
The tropical spaces of interest in this paper are real tori with integral structures which, in particular, belong to the category of boundaryless rational polyhedral spaces. In this section we briefly review the preliminaries following \cite{GSH, GSJ}. We refer the reader to these papers for details and further references.

\subsection{Boundaryless rational polyhedral spaces}\label{subsec:BRPS}
A rational polyhedral set in $\mathbb{R}^n$ is a finite union of finite intersections of sets of the form $\{x \in \mathbb{R}^n | \langle m, x\rangle \leq a\}$,
where $m\in(\mathbb{Z}^n)^\ast$, $a\in\mathbb{R}$, and $\langle \cdot, \cdot \rangle$ denotes the evaluation pairing. Any such $P$ is equipped with a sheaf $\Aff_P$ of integral affine functions. A \textit{boundaryless rational polyhedral space} is a pair $(X, \Aff_X)$ of a topological space $X$ and sheaf $\Aff_X$ of continuous real-valued functions such that locally around each point $x\in X$, the pair is isomorphic to the pair $(P, \Aff_P)$ for a rational polyhedral set $P$ in some $\mathbb{R}^n$. The sections of $\Aff_X$ are called \textit{integral affine functions} on $X$. If $X$ is compact, we call it a \textit{closed rational polyhedral space}. 

 A morphism of boundaryless rational polyhedral spaces is a continuous map $f : X\rightarrow Y$ of topological spaces such that pullbacks of functions in $\Aff_Y$ are in $\Aff_X$. A morphism $f : X\rightarrow Y$ is called proper if it is a proper map of topological spaces, that is, preimages of compact sets are compact. 

\subsection{Tropical cycles, cycle classes, and (co)homology}\label{subsec:tropcycles}

We briefly recall the theory of tropical cycles and tropical (co)homology, as well as the tropical cycle class map which connects the two notions. For details, we refer the reader to \cite{GSJ}, \cite{AR10}, \cite{FR13}, \cite{Shaw13} regarding tropical cycles and \cite{IKMZ}, \cite{MZ14}, \cite{JRS}, \cite{GSH} regarding tropical (co)homology and the tropical cycle class map.

\subsubsection{Tropical cycles}\label{subsubsec:cycle}
For a boundaryless rational polyhedral space $X$, let $X^{\text {reg }}$ denote the open subset of points $x \in X$ that have a neighborhood isomorphic (as boundaryless rational polyhedral spaces) to an open subset of $\mathbb{R}^n$, for some $n \in \mathbb{N}$. 

\begin{definition}\label{Def:tropcycle}
 A tropical $k$-cycle on $X$ is a function $A: X \rightarrow \mathbb{Z}$, with support $|A|=\overline{\{x \in X \mid A(x) \neq 0\}}$ being either empty or a purely $k$-dimensional polyhedral subset of $X$, such that $A$ is nonzero precisely on $|A|^{\text {reg }}$ where it is locally constant. Furthermore, $A$ must satisfy a so-called \textit{balancing condition} on $|A|^{\text{reg}}$ (see \cite{AR10} for details).  
\end{definition} 
 One can appropriately define the sum of two tropical $k$-cycles on $X$ which makes the set $Z_k(X)$ of tropical $k$-cycles on $X$ into an abelian group. A tropical cycle $A$ is defined to be \textit{effective} if it is everywhere nonnegative. Moreover, any proper morphism $f: X \rightarrow Y$ of boundaryless rational polyhedral spaces induces a \textit{pushforward} $f_*: Z_k(X) \rightarrow Z_k(Y)$ of tropical cycles, which is a homomorphism. If $A \in Z_k(X)$, then the support of the pushforward cycle $f_* A\in Z_k(Y)$ is the closure of the set $(f|A|)_k \subseteq f|A|$, where the local dimension of $f|A|$ is $k$. If $X$ is compact then the constant map $X\rightarrow Y=$ point, is proper. Identifying the tropical 0-cycles on a point with $\mathbb{Z}$, the pushforward then defines a homomorphism $\int_X: Z_0(X) \rightarrow \mathbb{Z}$ which we call the degree map.

A \textit{rational function} on a boundaryless rational polyhedral space $X$ is a continuous function $\phi: X \rightarrow \mathbb{R}$ such that $\phi$ is piecewise affine with integral slopes in every chart. As this is a local condition, rational functions define a sheaf $\mathscr{M}_X$ of abelian groups, into which the sheaf $\Aff_X$ naturally injects. 
\begin{definition}\label{Def:cartdiv}
   The group of \textit{tropical Cartier divisors} on $X$ is defined as: $$\operatorname{CDiv}(X)=\Gamma\left(X, \mathscr{M}_X / \operatorname{Aff}_X\right)\, .$$ 
\end{definition} 
For every $\phi \in \Gamma\left(X, \mathscr{M}_X\right)$, we denote its image in $\operatorname{CDiv}(X)$ by $\operatorname{div}(\phi)$, and refer to it as the associated \textit{principal Cartier divisor}. As in classical algebraic geometry, there exists a natural bilinear map $\operatorname{CDiv}(X) \times Z_k(X) \rightarrow$ $Z_{k-1}(X)$ given by the intersection pairing of divisors and tropical cycles. Note that a boundaryless rational polyhedral space $X$ does not automatically have a natural fundamental cycle, that is, there is no canonical element in $Z_\bullet(X)\coloneqq\oplus_{k=0}^{\dim X}Z_k(X)$ in general. When such a canonical cycle exists, it is defined as follows.
\begin{definition}\label{Def:fundcyc}
A boundaryless rational polyhedral space $X$ is said to \textit{have a fundamental cycle}, denoted by $[X]$, if $X$ is pure-dimensional and the extension by $0$ of the constant function with value $1$ on $X^{\text{reg }}$ defines a tropical cycle. 
\end{definition}
A Cartier divisor $D \in \operatorname{CDiv}(X)$ on a tropical space $X$ admitting a fundamental cycle is defined to be \textit{effective} if its \textit{associated Weil divisor} $[D]\coloneqq D \cdot[X]$ is effective. For $X$ a tropical manifold, there exists a fundamental cycle $[X]$ which serves as the unity of the \textit{tropical intersection product} on $Z_\bullet(X)$. This makes $Z_\bullet(X)$ into a ring. The tropical intersection product is compatible with intersections with Cartier divisors, that is, $D \cdot A=[D] \cdot A$ for every Cartier divisor $D \in \operatorname{CDiv}(X)$ and tropical cycle $A \in Z_\bullet(X)$. Moreover, the morphism:
$$
\operatorname{CDiv}(X) \rightarrow Z_{\operatorname{dim}(X)-1}(X)\, , \quad D \mapsto[D]
$$
is an isomorphism (see \cite[Corollary 4.9]{Fran13}), whereby one can use tropical Cartier divisors and codimension-$1$ tropical cycles on a tropical manifold interchangeably. 

\subsubsection{Line bundles}\label{subsec:linbundle}
A \textit{tropical line bundle} on a boundaryless rational polyhedral space $X$ is defined as an $\Aff_X$-torsor. Analogous to topological line bundles, a tropical line bundle $\mathscr{L}$ corresponds to a morphism $\pi_\mathscr{L}: Y \rightarrow X$ of boundaryless rational polyhedral spaces such that there are local trivializations of $\pi_\mathscr{L}$ on $X$, where two such trivializations are related via translation by an integral affine function. Analogous to the classical situation, a standard \v{C}ech cohomology argument shows that the set of isomorphism classes of tropical line bundles on $X$ is in natural bijection to the group $H^1\left(X, \mathrm{Aff}_X\right)$ of tropical \v{C}ech cocycles. Consequently, isomorphism classes of tropical line bundles form a group. 

A \textit{rational section} of a tropical line bundle $Y \rightarrow X$ is a continuous section, given by a rational function in all trivializations. As in classical algebraic geometry, there also exists a bijection between $\operatorname{CDiv}(X)$ and isomorphism classes of pairs $(\mathscr{L}, s)$, where $\mathscr{L}$ is a tropical line bundle on $X$ and $s$ is a rational section of $\mathscr{L}$ .

\subsubsection{Tropical (co)homology}\label{subsubsec:tropcohom}
We briefly recall the bigraded homology theory for tropical spaces introduced in \cite{IKMZ}. A sheaf theoretic approach to tropical (co)homology is given in \cite{GSH}, which is our main reference. Let $X$ be a boundaryless rational polyhedral space. To define the tropical homology and cohomology groups, one needs sheaves $\Omega_X^p$ of tropical $p$-forms for $p>0$. The cotangent sheaf $\Omega_X^1$ on $X$ is defined as the quotient sheaf $\Aff_X/\bR_X$, where $\bR_X$ is the constant sheaf associated to $\bR$. For $p>1$, the sheaves $\Omega_X^p$ are defined to be the image of the sheaf homomorphism $\bigwedge^p \Omega_X^1 \rightarrow \iota_*\left(\left.\bigwedge^p \Omega_X^1\right|_{X^{\mathrm{reg}}}\right)$, where $\iota: X^{\text{reg}}\hookrightarrow X$ is the natural inclusion. This is because $\Omega_X^p$ is the natural extension of the sheaf $\bigwedge^p \Omega_X^1$ on $X^{\text{reg}}$ to $X$, such that $\Omega_X^p=0$ for $p>\dim X$. As shown in \cite[Example~2.9]{GSH}, $\bigwedge^p\Omega_X^1$ defined globally on $X$ may fail to be $0$ even if $p>\dim X$. 

We refer the reader to \cite{IKMZ} and \cite{GSH} for the precise definitions of tropical homology $H_{p,q}(X)$ and cohomology $H^{p,q}(X)$ for a general tropical space $X$ using (co)chain complexes. We do not need the general theory in this paper. However, we note that there exists a natural isomorphism $H^{p, q}(X) \cong H^q\left(X, \Omega_X^p\right)$ of tropical cohomology with abelian sheaf cohomology.

\subsubsection{The first Chern class map and tropical cycle class map}\label{subsubsec:cherncycleclass} The quotient map $d: \operatorname{Aff}_X \rightarrow \Omega_X^1$ of sheaves on a boundaryless rational polyhedral space $X$ induces a morphism:
$$
c_1\coloneqq H^1(d): H^1\left(X, \operatorname{Aff}_X\right) \rightarrow H^1\left(X, \Omega_X^1\right) \cong H^{1,1}(X)\, ,
$$
called the \textit{first Chern class map} from the group of all tropical line bundles on $X$ to the $(1,1)$-tropical cohomology group of $X$. Using the first Chern class map, any divisor $D \in \operatorname{CDiv}(X)$ has an associated $(1,1)$-cohomology class $c_1(\mathscr{L}(D))$, where $\mathscr{L}(D)\in H^1(X, \Aff_X)$ is the line bundle associated to $D\in\operatorname{CDiv}(X)$.

As in classical algebraic geometry, there is a tropical cycle class map which assigns a class in tropical homology to every tropical cycle. More precisely, on any closed rational polyhedral space $X$, there exist morphisms:
$$
\operatorname{cyc}: Z_k(X) \rightarrow H_{k, k}(X)
$$
for every $k \in \mathbb{N}$. We refer the reader to \cite[\S3E]{GSJ} for an explicit description of the cycle class map for $k=1$. If $X$ is a closed tropical manifold, then tropical homology and cohomology are dual to each other by the following homomorphism, which turns out to be an isomorphism:
\[H^{*,*}(X)\rightarrow H_{*,*}(X)\, , \ \ c\mapsto c\frown \operatorname{cyc[X]}\, .\]
Here, $\frown$ denotes the cap product. The first Chern class map $c_1$ and the cycle class map for $k=\dim(X)-1$ are related via the following commutative diagram:
\begin{equation}\label{Eqn:c1cyc}
    \begin{tikzcd}
	{\operatorname{CDiv}(X)=Z_{\dim(X)-1}(X)} && {H_{\dim(X)-1,\dim(X)-1}(X)} \\
	{H^1(X,\Aff_X)} && {H^{1,1}(X)}
	\arrow["{\operatorname{cyc}}", from=1-1, to=1-3]
	\arrow[from=1-1, to=2-1]
	\arrow["{-\frown\operatorname{cyc}[X]}", from=1-3, to=2-3]
	\arrow["{c_1}"', from=2-1, to=2-3]
\end{tikzcd}
\end{equation}

\section{Real tori with integral structure}\label{sec:RealtoriPrelim}

Having briefly described the general preliminaries for this paper, we now record some preliminaries on the tropical geometry of real tori with integral structures, which shall play an important role in the rest of this paper. Again, we refer the reader to \cite{GSJ}, \cite[\S4]{RZ} for the details.

\subsection{Real tori and integral structures}\label{subsec:intrealtori}
Let $N$ be a lattice, and let $\Lambda \subseteq N_{\mathbb{R}}=N \otimes_{\mathbb{Z}} \mathbb{R}$ be a second lattice of full rank, that is equal to the rank of $N$. Any isomorphism $N \cong \mathbb{Z}^n$ induces a well-defined rational polyhedral structure on $N_\bR$. The \textit{real torus with integral structure} associated to $N$ and $\Lambda$ is the quotient $X=N_{\mathbb{R}} / \Lambda$, with the sheaf of integral affine functions being the one induced by $N_{\mathbb{R}}$, that is, if $\pi: N_{\mathbb{R}} \rightarrow X$ is the quotient map then $\Aff_X\coloneqq\pi_{*}\Aff_{N_\bR}$ is the pushforward of the sheaf of integral affine functions on the universal cover of $X$. The integral affine structure on $X$ is induced by $N$ and not by $\Lambda$.

Furthermore, $X$ has a natural group structure induced by the vector space addition on $N_\bR$, which is compatible with the integral structure. Thus, real tori with integral structures are group objects in the category of boundaryless rational polyhedral spaces.

\subsection{Tropical (co)homology of real tori with integral structures}\label{subsec:realtoricohom}

Let $X=N_{\mathbb{R}} / \Lambda$ be a real torus with integral structure. The group law and the tropical cross product endow the tropical homology groups of $X$ with the additional structure of the Pontryagin product.

\begin{definition}\label{Def:Pontryaginprod} Let $X$ be a real torus with integral structure with group law $\mu: X \times X \rightarrow X$. The tropical Pontryagin product is defined as the pairing:
$$
(\alpha, \beta) \mapsto \alpha \star \beta\coloneqq\mu_*(\alpha \times \beta)\, ,
$$
where $\alpha$ and $\beta$ are elements of $Z_*(X)$ (respectively, $H_{*, *}(X)$) and $\alpha\times\beta$ is an element of $Z_*(X\times X)$ (respectively, of $H_{*,*}(X\times X)$) denotes the cross product in cycles (respectively, homology). We obtain homomorphisms:
$$
\begin{gathered}
\star: Z_i(X) \otimes_{\mathbb{Z}} Z_j(X) \rightarrow Z_{i+j}(X)\, , \\
\star: H_{i, j}(X) \otimes_{\mathbb{Z}} H_{k, l}(X) \rightarrow H_{i+k, j+l}(X)\, ,
\end{gathered}
$$
for all choices of natural numbers $i, j, k, l$, which we call the Pontryagin product on tropical cycles and homology, respectively.
\end{definition}

By \cite[Proposition~6.2]{GSJ}, the Pontryagin product on cycles commutes with the Pontryagin product on tropical homology via the cycle class map. For real torus with integral structure $X=N_{\mathbb{R}} / \Lambda$, the tropical homology groups $H_{*, *}(X)$ and the Pontryagin product can be described explicitly. By observing that $\Omega^p_X$ equals the constant sheaf $(\bigwedge^pM)_X$ for $M\coloneqq\Hom(N,\bZ)=N^*$, one obtains the following canonical isomorphism for all $p,q\geq 0$:
\begin{equation}\label{Eqn:Torushom}
H_{p, q}(X) \cong H_q\left(X ; \bigwedge^p N\right) \cong H_q(X ; \mathbb{Z}) \otimes_{\mathbb{Z}} \bigwedge^p N \cong \bigwedge^q \Lambda \otimes_{\mathbb{Z}} \bigwedge^p N.
\end{equation}

We refer the reader to \cite[\S6]{GSJ} for the details of the above isomorphism. It is straightforward to check that, with the identification~\eqref{Eqn:Torushom}, the tropical Pontryagin product on $H_{\bullet, \bullet}(X)$ satisfies $(\alpha \otimes \omega) \star(\beta \otimes \xi)=(\alpha \wedge \beta) \otimes(\omega \wedge \xi)$. Similarly, one obtains the following explicit description for the tropical cohomology of $X$, which is dual to the description of tropical homology in \eqref{Eqn:Torushom}:
\begin{equation}\label{Eqn:Toruscohom}
H^{p, q}(X) \cong \bigwedge^q \Lambda^* \otimes_{\mathbb{Z}} \bigwedge^p N^*\, .
\end{equation}
With this identification, the tropical cup product on $H^{\bullet, \bullet}(X)$ satisfies $(\alpha \otimes \omega) \smile(\beta \otimes \xi)=(\alpha \wedge \beta) \otimes(\omega \wedge \xi)$. 

By the descriptions of the tropical homology and cohomology given in \eqref{Eqn:Torushom} and \eqref{Eqn:Toruscohom}, the tropical cap product pairing between homology and cohomology can be explicitly expressed by $(\alpha \otimes \omega) \frown(\beta \otimes \xi)=(\alpha\lrcorner \beta) \otimes(\omega\lrcorner \xi)$,
where " $\lrcorner$ " denotes the interior product on the exterior algebra.

In bidegree $(1,1)$ the description \eqref{Eqn:Toruscohom} produces an isomorphism $H^{1,1}(X) \cong \Lambda^* \otimes_{\mathbb{Z}} N^*$, which we can further identify with $\Hom(\Lambda \otimes_{\mathbb{Z}} N, \mathbb{Z})$. So, an element of $H^{1,1}(X)$ can ve viewed as a bilinear form on $N_{\mathbb{R}}$ having integer values on $\Lambda \times N$. We shall routinely use this interpretation in the rest of the paper.

\subsection{Line bundles on real tori with integral structures} \label{subsec:realtorilinebundle}

 We now briefly describe the theory of tropical line bundles on real tori with integral structures, following \cite{GSJ}. This will be used in Section~\ref{sec:Poincarebundle} in the construction of the Poincar\'e bundle, which is the central object in Fourier--Mukai theory.

\subsubsection{Tropical factors of automorphy}\label{subsubsec:autfactor} As before, let $X=N_{\mathbb{R}} / \Lambda$ be a real torus with integral structure. As described in \S\ref{subsec:linbundle} tropical line bundles on $X$ form a group, canonically identified with $H^1\left(X, \mathrm{Aff}_X\right)$. In this case, furthermore, there is a canonical isomorphism $H^1\left(X, \operatorname{Aff}_X\right) \cong H^1\left(\Lambda, \Gamma\left(N_{\mathbb{R}}, \operatorname{Aff}_{N_{\mathbb{R}}}\right)\right)$ of $H^1(X,\Aff_X)$ with the first cohomology group of $\Gamma(N_\bR, \operatorname{Aff}_{N_\bR})$ with its natural $\Lambda$-action (see \cite[\S7A]{GSJ}).
 An element of $H^1\left(\Lambda, \Gamma\left(N_{\mathbb{R}}, \Aff_{N_{\mathbb{R}}}\right)\right)$ is represented by a \textit{tropical factor of automorphy}, which is a function $a: \Lambda \times N_{\mathbb{R}} \rightarrow \mathbb{R}$, satisfying:
\begin{equation}\label{Eqn:autfact}
a(\lambda+\mu, x)=a(\lambda, \mu+x)+a(\mu, x)
\end{equation}
for all $\mu, \lambda \in \Lambda$ and $x \in N_{\mathbb{R}}$. Two factors of automorphy represent the same element of $H^1\left(\Lambda, \Gamma\left(N_{\mathbb{R}}, \operatorname{Aff}_{N_{\mathbb{R}}}\right)\right)$ if and only if their difference is a factor of automorphy of the form $(\lambda, x) \mapsto m_{\mathbb{R}}(\lambda)$, where $m_{\mathbb{R}}$ is the $\mathbb{R}$-linear extension of a $\mathbb{Z}$-linear form $m: N \rightarrow \mathbb{Z}$.
Any factor of automorphy $a(-,-)$ defines a group action of $\Lambda$ on the trivial line bundle $N_{\mathbb{R}} \times \mathbb{R}$ on $N_{\mathbb{R}}$, defined by $\lambda .(x, b)=(x+\lambda, b+$ $a(\lambda, x))$. The tropical line bundle on $X$ corresponding to $a(-,-)$ is the quotient $\left(N_{\mathbb{R}} \times \mathbb{R}\right) / \Lambda$ via the above action.

\subsubsection{Tropical Appell--Humbert theorem}\label{subsubsec:Appell--Humbert} For every homomorphism $l \in \Hom(\Lambda, \bR)$ and every symmetric bilinear form $E$ on $N_\bR$ such that $E(\Lambda \times N) \subseteq \bZ$, we obtain a tropical factor of automorphy given by:
\begin{equation}\label{Eqn:Factaut}
a_{E,l} (\lambda, x) = l(\lambda) -E(\lambda, x) -\frac{1}{2}E(\lambda, \lambda)\, .
\end{equation}
 Let $\sL(E, l)$ denote the associated tropical line bundle $(N_\bR\times \bR)/\Lambda$ on $X$ defined by the $\Lambda$-action on $N_\bR\times\bR$ by $a_{E,l}$. By \cite{GSJ}*{Proposition~7.1}, the first Chern class $c_1(\sL(E, l))$ of $\sL(E, l)$ is equal to $E$, after identifying $H^{1,1}(X)$ with $\Hom(\Lambda\otimes_\bZ N, \bZ)$. Similar to the theory of line bundles on complex tori, there is a tropical Appell--Humbert theorem classifying tropical line bundles on real tori with integral structures (see  \cite{GSJ}*{Theorem 7.2}).

\begin{theorem}[tropical Appell--Humbert]
\label{thm:tropical-appell-humbert}
    Let $\sL$ be a tropical line bundle on the real torus with integral structure $X = N_{\bR} / \Lambda$. Then there exists $l \in \Hom(\Lambda, \bR)$ and a symmetric form $E$ on $N_{\bR}$ with $E(\Lambda \times N) \subseteq \bZ$ such that $\sL \cong \sL(E, l).$ Moreover, given another choice $l' \in \Hom(\Lambda, \bR)$ and symmetric form $E'$ on $N_{\bR}$ with $E'(\Lambda \times N) \subseteq \bZ$, we have $\sL \cong \sL(E', l')$ if and only if $E = E'$ and the linear form $(l-l')_{\bR}: N_{\bR} \rightarrow \bR$ has integer values on $N$.
\end{theorem}

\begin{remark}\label{rem:picgrp}
    The group law $\otimes$ on the set $\operatorname{Pic}(X)$ of isomorphism classes of tropical line bundles on $X=N_\bR/\Lambda$ can be described using Theorem~\ref{thm:tropical-appell-humbert} as:
    \[
    \sL(E_1, l_1)\otimes \sL(E_2, l_2)\cong\sL(E_1+E_2, l_1+l_2)\, .
    \]
\end{remark}

\subsection{Dual real tori with integral structures}\label{subsec:dualrealtori}
In this section we study some functorial properties of the dual $\widehat{X}=\Lambda_{\mathbb{R}}^{\ast}/N^{\ast}$ of a real torus with integral structure $X=N_{\mathbb{R}}/\Lambda$ .
By Theorem~\ref{thm:tropical-appell-humbert}, there is a bijection between the group of all tropical line bundles with trivial first Chern class and $\dual{X}$, whereby one can identify $\widehat{X}\cong \pic^0(X)$ and equip $\pic^0(X)$ with integral structure. 

For any $x\in N_{\mathbb{R}}$ let $\overline{x}\coloneqq\pi(x)$, where $\pi: N_{\mathbb{R}}\rightarrow X$ is the natural quotient map. Let $t_{\overline{x}}: X\rightarrow X$ denote the translation by $\overline{x}$ map on $X$.  

\begin{proposition}[Tropical theorem of the square]\label{prop:thmofsq}
    Let $t_{\overline{x}}^{*}\sL$ denote the pullback of any tropical line bundle $\sL$ on $X$ along the translation map $t_{\overline{x}}$. Then, for all $\overline{x}, \ \overline{y}\in X=N_{\mathbb{R}}/\Lambda$ and $\sL\in \operatorname{Pic}(X)$, we have:
    \[t^{*}_{\overline{x}+\overline{y}}\sL=t^{*}_{\overline{x}}\sL\otimes t^{*}_{\overline{y}}\sL\otimes \sL^{-1}\, .\]
\end{proposition}

\begin{proof}
    By Theorem~\ref{thm:tropical-appell-humbert}, $\sL\cong \sL(E, l)$ for some $l \in \Hom(\Lambda, \bR)$ and a symmetric bilinear form $E$ on $N_{\bR}$ with $E(\Lambda \times N) \subseteq \bZ$. Let $x, y \in N_{\mathbb{R}}$ such that $\pi(x)=\overline{x}$ and $\pi(y)=\overline{y}$. By \cite{GSJ}*{Proposition 7.5} and Remark~\ref{rem:picgrp}, we have
    \begin{align*}
        t^{\ast}_{\overline{x}+\overline{y}}\sL\cong \sL(E, l-E(-, x+y))=\sL(E, l-E(-,x)-E(-,y))\\ =\sL(E+E-E, l-E(-,x)+l-E(-,y)-l)\\\cong t^{\ast}_{\overline{x}}\sL(E,l)\otimes t^{\ast}_{\overline{y}}\sL(E, l)\otimes \sL(-E,-l)\\\cong t^{\ast}_{\overline{x}}\sL\otimes t^{\ast}_{\overline{y}}\sL\otimes \sL^{-1}\, .
    \end{align*}
\end{proof}
By \cite[Proposition 7.5]{GSJ}, one can check that for any $\overline{x}\in X$ and $\sL\in \operatorname{Pic}(X)$, the line bundle $t^{\ast}_{\overline{x}}\sL\otimes \sL^{-1}$ has trivial first Chern class. Thus, for any $\sL\in\operatorname{Pic}(X)$, we have a map $\phi_{\sL}: X\rightarrow \widehat{X}$ given by $\overline{x}\mapsto t^{\ast}_{\overline{x}}\sL\otimes \sL^{-1}$. By Proposition~\ref{prop:thmofsq}, it follows that $\phi_{\sL}$ is also a group homomorphism for any $\sL$. So therefore it is a morphism of real tori with integral structures. The following lemma shows $\phi_{\sL}: X\rightarrow \dual{X}$ can always be lifted to a homomorphism $N_{\bR}\rightarrow \Lambda^{\ast}_{\bR}$ of the universal covers of the tori.

\begin{lemma}\label{lem:analyticrep}
    For $\sL=\sL(E,l)\in \operatorname{Pic}(X)$, the morphism $\phi_{\sL}: X\rightarrow \widehat{X}$ is induced by a linear map $\Phi_E: N_{\mathbb{R}}\rightarrow \Lambda^{\ast}_{\mathbb{R}}$ on the universal covers defined by $\Phi_E(x)=-E(-,x)$.
\end{lemma}

\begin{proof}
    Let $\pi: N_{\mathbb{R}}\rightarrow X$ and $\pi^{\ast}: \Lambda^{\ast}_{\mathbb{R}}\rightarrow \widehat{X}$ be the natural projection maps from the respective universal covers.  Since $E$ is a symmetric bilinear form on $N_{\mathbb{R}}=\Lambda_{\mathbb{R}}$ such that $E(\Lambda\times N)\subseteq\bZ$, the map $\Phi_{E}: N_{\mathbb{R}}\rightarrow \Lambda^{\ast}_{\mathbb{R}}$ is a $\mathbb{R}$-linear map. If $x\in \Lambda$ we have $\Phi_E(x)=-E(-,x)=-E(x,-)\in N^{\ast}$. Therefore, $\Phi_E$ descends to a map $X\rightarrow \widehat{X}$. Finally, by definition, we have: $$\pi^{\ast}(\Phi_{E}(x))=\pi^{\ast}(-E(-,x))=\sL(0, -E(-,x))=\phi_{\sL}(\overline{x})=\phi_{\sL}(\pi(x))\, .$$
\end{proof}

\begin{remark}\label{rem:isonondegen}\hfill
    \begin{itemize} 
        \item[(i)] Let $K(\sL)\coloneqq\ker\phi_{\sL}\subseteq X$. Then $K(\sL)=\Lambda(\sL)/\Lambda$, where $\Lambda(\sL)\coloneqq\{x\in N_{\mathbb{R}}\mid \ E(N,x)\subseteq \bZ\}=\Phi_E^{-1}(N^{\ast})$. Using Theorem~\ref{thm:tropical-appell-humbert}, it is straightforward to verify that $K(\sL)$ satisfies properties analogous to \cite[Lemma 2.4.7]{CAV}.
        \item[(ii)]  The map $\phi_{\sL}:X\rightarrow \widehat{X}$ has finite kernel if and only if $\Lambda\subseteq\Lambda(\sL)$ is a finite index sublattice. Since $\Lambda\subset N_\bR$ has full rank, this is equivalent to $\Lambda(\sL)\subset N_\bR$ being a lattice of full rank. This is also equivalent to $E$ being a nondegenerate bilinear form on $N_\bR$. In this case, we define degree of $\phi_\sL$ to be the index $\deg \phi_\sL\coloneqq[\Lambda(\sL):\Lambda]$.
    \end{itemize}
\end{remark}

\subsection{Tropical abelian varieties}\label{subsec:tropAV}

Let $X=N_{\mathbb{R}}/\Lambda$ be a real torus with integral structure. In analogy to complex abelian varieties, a polarization on $X$ is, by definition, the first Chern
class $E = c_1(\sL)$ of a \textit{positive-definite line bundle} $\sL$ on $X$, meaning its first Chern class is positive-definite as a bilinear form. We sometimes refer to the positive-definite line bundle $\sL$  as the polarization. It is well known that a line bundle on a complex abelian variety is ample if and only if it is a positive-definite (see \cite[Proposition 4.5.2]{CAV}). This motivates the following definitions.

\begin{definition}\label{def:tropample}
    A tropical line bundle $\sL(E, l)$ on a real torus with integral structure $X=N_{\mathbb{R}}/\Lambda$  is said to be \textit{ample} if it is a positive-definite line bundle, that is, $c_1(\sL)=E$ is a positive-definite bilinear form on $N_{\mathbb{R}}$.
\end{definition}

\begin{definition}\label{def:tropAV}
A \textit{tropical abelian variety} is defined as a real torus with integral structure $X$ , admitting a polarization $E = c_1(\sL)$. The pair $(X,E)$ (or sometimes written as $(X, \sL)$) is called a \textit{polarized tropical abelian variety}.  Alternatively, a tropical abelian variety is a real torus with integral structure $X$  admitting an ample tropical line bundle $\sL$.
\end{definition}

Note that for any line bundle $\sL(E, l)$, the first Chern class, $c_1(\sL(E, l))=E$ is an element of $H^{1,1}(X)=\Lambda^{\ast}\otimes_{\mathbb{Z}}N^{\ast}\cong \Hom(\Lambda, N^{\ast})$, whereby $E$ induces a homomorphism of lattices $\rho_E: \Lambda\rightarrow N^{\ast}$ defined by $\rho_E(\lambda)=E(\lambda,-)\in N^{\ast}$ for all $\lambda\in \Lambda$. We define a tropical line bundle $\sL$ to be \textit{nondegenerate} if its first Chern class $c_1(\sL)=E$ is nondegenerate as a bilinear form. For a nondegenerate (and, in particular, positive-definite) line bundle $\sL(E, l)$, the induced lattice homomorphism $\rho_{E}:\Lambda\rightarrow N^{\ast}$ is injective, because it induces an isomorphism $(\rho_E)_{\bR}: \Lambda_\bR\rightarrow N^{\ast}_\bR$ after tensoring with $\mathbb{R}$. By Remark~\ref{rem:isonondegen} (ii), the map $\phi_\sL:X\rightarrow \widehat{X}$ has finite degree if and only if $\sL$ is a nondegenerate line bundle. A \textit{principally polarized tropical abelian variety} is a polarized abelian variety $(X, E)$, such that the induced lattice homomorphism $\rho_E: \Lambda\rightarrow N^{\ast}$ is an \textit{isomorphism of lattices}.

\section{Tropical Poincar\'e bundle}\label{sec:Poincarebundle}
For $X$ a real torus with integral structure, the dual torus $\dual{X}$, by definition, parametrizes the isomorphism classes of line bundles in $\pic^0(X)$. Analogous to complex abelian varieties, one would hope for a tropical line bundle on the product $X\times \dual{X}$ inducing all the line bundles of $\pic^0(X)$. Indeed, we now define the \textit{tropical Poincar\'e bundle} on the product $X\times \dual{X}$, which serves the purpose of such a universal line bundle. 

\begin{definition}\label{def:Poincare}
    Let $X=N_{\mathbb{R}}/\Lambda$ be a real torus with integral structure, and let $\dual{X}$ be its dual. A tropical Poincar\'e bundle $\mathcal{P}_X$ on $X\times \dual{X}$ is a tropical line bundle on the product torus, satisfying:
    \begin{enumerate}
        \item $\mathcal{P}_X|_{X\times\{L\}}\simeq L$ for every $L\in\dual{X}$, and
        \item $\mathcal{P}_X|_{\{0\}\times \dual{X}}$ is trivial.
    \end{enumerate}
\end{definition}

\textit{A priori}, the existence of such a line bundle $\mathcal{P}_X$ on $X\times\dual{X}$ is not evident from the definition. We will show that there indeed exists such a line bundle, unique up to isomorphism. Subsequently, we call this \textit{the Poincar\'e bundle} $\mathcal{P}_X$ on $X\times\dual{X}$. The following lemma establishes a weak form of the \textit{seesaw principle} for real tori with integral structures, which shall be useful in proving the uniqueness of the Poincar\'e bundle.

\subsection{A weak seesaw principle for real tori with integral structures}\label{subsec:seesaw}
In this section we prove a weak version of a tropical analogue of the classical seesaw theorem for algebraic varieties (see \cite[\S5]{MumAV}). The seesaw theorem in its full generality is proved using sheaf cohomology and base change. We circumvent this for real tori with integral structures via an application of the tropical Appell--Humbert theorem.

\begin{lemma}[A weak seesaw principle]\label{lem:seesaw}
    Let $X=N_{\mathbb{R}}/\Lambda$ and $Y=N'_{\mathbb{R}}/\Lambda'$ be real tori with integral structures. Let $\sL\in \pic(X\times Y)$ be a tropical line bundle such that:
    \begin{enumerate}
        \item[(a)] $\sL|_{X\times\{y\}}$ is trivial for every $y\in Y$, and
        \item[(b)] $\sL|_{\{0\}\times Y}$ is trivial.
    \end{enumerate}
    Then $\sL$ is the trivial line bundle on $X\times Y$.
\end{lemma}

\begin{proof}
    Note that $X\times Y= (N_\bR\oplus N'_\bR)/(\Lambda\oplus\Lambda')$ is a real torus with integral structure. By Theorem~\ref{thm:tropical-appell-humbert}, there exists $l\in \Hom(\Lambda\oplus \Lambda', \bR)$ and a symmetric bilinear form $E$ on $N_\bR\oplus N'_\bR$ with $E((\Lambda\oplus \Lambda')\times (N\oplus N'))\subseteq \mathbb{Z}$ such that $\sL=\sL(E, l)$. Let $a_{E,l}:(\Lambda\oplus\Lambda')\times(N_\bR\oplus N'_\bR)\rightarrow \bR$ be the corresponding tropical factor of automorphy of $\sL$; for $(\lambda, \lambda')\in \Lambda\oplus\Lambda'$ and $(x, x')\in N_\bR\oplus N'_\bR$:
    \[ a_{E,l}((\lambda,\lambda'),(x, x'))=l(\lambda, \lambda')-E((\lambda,\lambda'),(x, x'))-\frac{1}{2}E((\lambda,\lambda'),(\lambda, \lambda'))\, .\]
    For any $y\in Y$, the restriction $\sL|_{X\times\{y\}}$ is given by the factor $a_{E,l}|_{(\Lambda\oplus \{0\})\times(N_\bR\oplus\{y\})}$:
    \begin{equation}\label{Eq1}
         a_{E,l}((\lambda,0),(x, y))=l(\lambda, 0)-E((\lambda,0),(x, y))-\frac{1}{2}E((\lambda,0),(\lambda, 0))\, .
    \end{equation}
    Let $\tilde{l}\coloneqq l(-,0)\in \Hom(\Lambda, \bR)$ and $E_y(x_1, x_2)\coloneqq E((x_1,0),(x_2,y))$ for $x_1, x_2\in N_\bR$. Since $\sL|_{X\times\{y\}}$ is trivial for all $y\in Y$ by (a), it follows that $a_{E,l}((\lambda,0),(x, y))=0$ for all $\lambda \in \Lambda$, $x\in N_\bR$ and $y\in N'_\bR$.  By rearranging \eqref{Eq1}, it follows that $E_y$ is independent of $y$ due to triviality of  $\sL|_{X\times\{y\}}$. Consequently, we have $E_y(x_1,x_2)=E_0(x_1, x_2)=E_0(x_2, x_1)=E_y(x_2, x_1)$. It follows that $E_y$ is a symmetric bilinear form on $N_\bR$ such that $E_y(\Lambda, N)\subseteq \bZ$. For $\lambda\in \Lambda$ and $x\in N_\bR$, we can rewrite \eqref{Eq1} as:
    \begin{equation}\label{Eq2}
        0=\tilde{l}(\lambda)-E_y(\lambda, x)-\frac{1}{2}E_y(\lambda, \lambda)\, .
    \end{equation}
    Thus, the trivial line bundle on $X$ is isomorphic to $\sL(\tilde{l}, E_y)$. Since the trivial line bundle is given by $\sL(0, \mathbf{0})$, where $0\in \Hom(\Lambda, \bR)$ is the zero map and $\mathbf{0}$ is the trivial bilinear form on $N_\bR$, by Theorem~\ref{thm:tropical-appell-humbert}, we see that $\tilde{l}$ is a $\bR$-linear extension of a linear form $N\rightarrow \bZ$ and $E_y=\mathbf{0}$ for all $y\in Y$.

    The restriction $\sL|_{\{0\}\times Y}$ is given by the factor $a_{E,l}|_{(\{0\}\oplus \Lambda')\times(\{0\}\oplus N'_\bR)}$:
    \begin{equation}\label{Eq3}
        a_{E,l}((0,\lambda'),(0, x'))=l(0, \lambda')-E((0,\lambda'),(0, x'))-\frac{1}{2}E((0,\lambda'),(0, \lambda'))\, .
    \end{equation}
    The triviality of $\sL|_{\{0\}\times Y}$ by (b) implies that the left hand side of \eqref{Eq3} is $0$. Letting $\tilde{l}'\in \Hom(\Lambda', \bR)$ defined by $\tilde{l}'(-)\coloneqq l(0,-)$ and $E'_0(x_1', x_2')\coloneqq E((0, x_1'), (0, x_2'))$ for $x_1', x_2'\in N'_\bR$, a similar argument as above shows that $\tilde{l}'$ is a $\bR$-linear extension of a linear form $\Lambda'\rightarrow \bZ$ and $E'_0=\mathbf{0}'$, the trivial bilinear form on $N'_\bR$. Since $l\in \Hom(\Lambda\oplus \Lambda', \bR)$ is equal to $\tilde{l}+\tilde{l}'$, we see that $l$ is a $\bR$-linear extension of a linear form $\Lambda\oplus \Lambda'\rightarrow \bZ$. The triviality of the bilinear forms $E_y$ on $N_\bR$ for all $y\in Y$ and $E'_0$ on $N'_{\bR}$ implies that $E$ is the zero bilinear form on $N_\bR\oplus N'_\bR$. Thus, $\sL=\sL(l, E)=\sL(0, \mathbf{0})$. Therefore $\sL$ is isomorphic to the trivial line bundle on $X\times Y$ by another application of Theorem~\ref{thm:tropical-appell-humbert}.
\end{proof}

\subsection{Construction of the tropical Poincar\'e bundle}\label{subsec:Poincareconstruction}

We next prove the existence and uniqueness of the tropical Poincar\'e bundle on $X\times\widehat{X}$.
\begin{theorem}\label{Poincare:existence}
There exists a tropical Poincar\'e bundle $\mathcal{P}_X$ on $X\times \dual{X}$, uniquely determined up to isomorphism.

\end{theorem}

\begin{proof}
    \textit{Existence}: The product $X\times\dual{X}$ is given by the quotient $(N_{\bR}\oplus \Lambda^{\ast}_\bR)/(\Lambda\oplus N^{\ast})$, with integral structure given by the lattice $N\oplus \Lambda^{\ast}\subseteq N_{\bR}\oplus \Lambda^{\ast}_\bR$. Define a symmetric bilinear form $E: (N_\bR\oplus \Lambda^{\ast}_\bR)\times(N_\bR\oplus \Lambda^{\ast}_\bR)\rightarrow \bR$ as follows:
    \[ E((n_1, \lambda_1), (n_2,\lambda_2))=-\lambda_2(n_1)-\lambda_1(n_2)\, .\]
    Note that $E((\Lambda\oplus N^{\ast})\times (N\oplus \Lambda^{\ast}))\subseteq \mathbb{Z}$ since $\Lambda^{\ast}\coloneqq\Hom_{\bZ}(\Lambda, \bZ)$ and, similarly for $N^{\ast}$. Let $l=\mathbf{0}\in \Hom(\Lambda\oplus N^{\ast}, \bR)$ be the zero map. By Theorem~\ref{thm:tropical-appell-humbert}, there exists a line bundle $\mathcal{P}_X\coloneqq\sL(E, 0)$ on $X\times\dual{X}$, with first Chern class $c_1(\mathcal{P}_X)=E$. We claim that $\mathcal{P}_X$ is a construction of the desired Poincar\'e bundle. For this, we need to check the properties (a) and (b) of Definition~\ref{def:Poincare}. The tropical factor of automorphy $a_{\mathcal{P}_X}:(\Lambda\oplus N^{\ast})\times (N_\bR\oplus \Lambda^{\ast}_\bR)\rightarrow \bR$ of $\mathcal{P}_X$ is given by:
    \begin{align*}
    a_{\mathcal{P}_X}((\lambda, \nu^{\ast}), (n, l^{\ast}))=l(\lambda, \nu^{\ast})-E((\lambda, \nu^{\ast}), (n, l^{\ast}))-\frac{1}{2}E((\lambda, \nu^{\ast}), (\lambda, \nu^{\ast}))\\= 0+l^{\ast}(\lambda)+\nu^{\ast}(n)+\frac{1}{2}\nu^{\ast}(\lambda)+\frac{1}{2}\nu^{\ast}(\lambda)\\=l^{\ast}(\lambda)+\nu^{\ast}(n)+\nu^{\ast}(\lambda)\, .
    \end{align*}
    To check Definition~\ref{def:Poincare}(a), for $L\in \dual{X}=\Lambda^{\ast}_\bR/N^{\ast}$, let $l^{\ast}_{L}\in \Lambda^{\ast}_{\bR}$ be a representative of $L$. The tropical factor of automorphy of the line bundle $\mathcal{P}_X|_{X\times\{L\}}$ is given by:
     \begin{align*}
        a_{\mathcal{P}_X}((\lambda, 0^{\ast}), (n, l^{\ast}_{L}))=l^{\ast}_{L}(\lambda)+0^{\ast}(n)+0^{\ast}(\lambda)=l^{\ast}_{L}(\lambda)=a_{0,l^{\ast}_L}(\lambda, n)\, ,
    \end{align*}
    which defines the line bundle $\sL(0, l^{\ast}_L)$ on $X\cong X\times\{L\}$. Since $l^{\ast}_L\in\Lambda^*_\bR$ is a representative of $L$, Theorem~\ref{thm:tropical-appell-humbert} implies $\sL(0, l^{\ast}_L)\cong L$, which establishes $\mathcal{P}_X|_{X\times\{L\}}\cong L$ for any $L\in \dual{X}$. This argument is independent of the choice of the representative $l^{\ast}_L$ of $L$, as any such two representative differ by an element of $N^*$ and, hence, represent the same line bundle by Theorem~\ref{thm:tropical-appell-humbert}.

    The tropical factor of automorphy of the restriction $\mathcal{P}_X|_{\{0\}\times \dual{X}}$ is given by $$a_{\mathcal{P}_X}((0, \nu^{\ast}), (0, l^{\ast}))=l^{\ast}(0)+\nu^{\ast}(0)+\nu^{\ast}(0)=0\, ,$$ which defines the trivial line bundle $\sL(0,0)$ on $\dual{X}\cong \{0\}\times \dual{X}$, thereby establishing Definition~\ref{def:Poincare}(b).\\

    \noindent\textit{Uniqueness}: Let $\mathcal{P}$ and $\mathcal{P}'$ be two candidates for a Poincar\'e bundle on $X\times \dual{X}$ satisfying Definition~\ref{def:Poincare}. Then $\mathcal{P}\otimes \mathcal{P}'^{-1}$ satisfies the hypothesis of Lemma~\ref{lem:seesaw} on $X\times\dual{X}$. 
\end{proof}

\subsection{First Chern class of the Poincar\'e bundle}\label{subsec:chernPoincare}

Let  $X=N_\bR/\Lambda$ be a real torus with integral structure with the Poincar\'e bundle by $\mathcal{P}_X\in \operatorname{Pic}(X\times\widehat{X})$. The first Chern class $E_X\coloneqq c_1(\mathcal{P}_X)$ of the Poincar\'e bundle is an element of $H^{1,1}(X\times\widehat{X})$ which can be explicitly identified with each of the following:
\begin{equation}\label{Eqn:ProdH11}
(\Lambda\oplus N^{\ast})^{\ast}\otimes (N\oplus \Lambda^{\ast})^{\ast}\cong \Hom(\Lambda\oplus N^{\ast}, (N\oplus \Lambda^{\ast})^{\ast})\cong\End(\Lambda\oplus N^{\ast})\, .
\end{equation}
Since $X\times\widehat{X}$ is a smooth compact pure-dimensional tropical manifold, the notions of Borel-Moore (co)homology and tropical (co)homology (as defined in \cite{GSH}) coincide for $X\times\widehat{X}$. Since the (co)homology groups of real tori with integral structures are torsion-free, by K\"unneth decomposition (\cite{GSH}*{Theorem B}) and Poincar\`e duality (\cite{GSH}*{Theorem D}), we obtain the decomposition:
\begin{equation}\label{eq:kunneth}
H^{1,1}(X\times\widehat{X})=\bigoplus_{0\leq p,q\leq 1}H^{p,q}(X)\otimes H^{1-p,1-q}(\widehat{X})\, .
\end{equation}
More concretely, $H^{1,1}(X\times \widehat{X})= (\Lambda^{\ast}\otimes N^{\ast})\oplus (\Lambda^{\ast}\otimes \Lambda)\oplus (N^{\ast}\otimes N)\oplus (N\otimes \Lambda)$. The following lemma shows that $c_1(\mathcal{P}_X)$ is contained in the middle two terms of the K\"unneth decomposition~\eqref{eq:kunneth}.

\begin{lemma}\label{lem:poincarecohom}
    $c_1(\mathcal{P}_X)\in (H^{1,0}(X)\otimes H^{0,1}(\widehat{X}))\oplus (H^{0,1}(X)\otimes H^{1,0}(\widehat{X}))=(N^{\ast}\otimes N)\oplus (\Lambda^{\ast}\otimes \Lambda)$.
\end{lemma}

\begin{proof}
    This is immediate from the form of $c_1(\mathcal{P}_X)=E_x: (N_\bR\oplus \Lambda^{\ast}_\bR)\times(N_\bR\oplus \Lambda^{\ast}_\bR)\rightarrow \bR$ defined in the proof of Theorem~\ref{Poincare:existence}:
    \[E_X((n_1, \lambda_1), (n_2,\lambda_2))=-\lambda_2(n_1)-\lambda_1(n_2)\, .\]
    Since $N_\bR=\Lambda_\bR$ (and, analogously, for the duals), the bilinear form $E_X$ is the linear map:
    \[(n_1,\lambda_1)\otimes (n_2, \lambda_2)\mapsto (n_2\otimes \lambda_1, n_1\otimes\lambda_2)\mapsto -\lambda_1(n_2)-\lambda_2(n_1)\, .\]
    The final map above is the sum of the negatives of the $\bR$-linear extensions of the natural evaluation pairings $N\otimes N^{\ast}\rightarrow \bZ$ and $\Lambda\otimes \Lambda^{\ast}\rightarrow \bZ$, since $E_X((\Lambda\oplus N^{\ast})\times (N\oplus \Lambda^{\ast}))\subseteq \mathbb{Z}$. Thus, $E_X\in (N\otimes N^{\ast})^{\ast}\oplus (\Lambda\otimes \Lambda^{\ast})^{\ast}$.
\end{proof}

\begin{remark}\hfill
    \begin{itemize}
        \item [(i)] Since there is a natural identification $H^{1,1}(X\times\widehat{X})\cong \End(\Lambda\oplus N^{\ast})$ by \eqref{Eqn:ProdH11}, the element $c_1(\mathcal{P}_X)$ defines a lattice endomorphism of $\Lambda\oplus N^{\ast}$.
        \item [(ii)] We have $c_1(\mathcal{P}_X)\in \End(\Lambda)\oplus\End(N^*)\subset \End(\Lambda\oplus N^*)$. Thus, we can write \[c_1(\mathcal{P}_X)=\Psi_{\Lambda}+\Psi_{N^*}, \text{ where } \Psi_{\Lambda}\in \End(\Lambda), \ \Psi_{N^*}\in \End(N^*)\, .\] 
    \end{itemize}
\end{remark}

Let $\Psi_N:N\rightarrow N$ be the dual of $\Psi_{N^*}$, i.e, the image of $\Psi_{N^*}$ under the canonically isomorphism $\End(N^*)\cong \End(N)$. Consequently, for any integers $p,q\geq 0$,  we see that $c_1(\mathcal{P}_X)$ induces natural homomorphisms: 
\begin{equation}\label{Eqn:indhom}
\bigwedge^q\Psi_{\Lambda}\otimes \bigwedge^p\Psi_N: H^{p,q}(X)^\ast \rightarrow H^{q, p}(\widehat{X})\, .
\end{equation} 

\begin{lemma}\label{lem:map}
    $\Psi_{\Lambda}\in \End(\Lambda)$ and $\Psi_{N}\in \End(N)$ are the multiplication by $-1$ maps. 
\end{lemma}

\begin{proof}
Both $\Psi_{\Lambda}$ and $\Psi_{N^{\ast}}$ correspond to the negative of the respective perfect evaluation pairings $\Lambda^{\ast}\otimes \Lambda\rightarrow \bZ$ and $N^{\ast}\otimes N\rightarrow \bZ$. Thus, if $\{ \lambda_i\}_{1\leq i\leq g}$ is a basis for $\Lambda$ with corresponding dual basis $\{\lambda^{\ast}_i\}_{1\leq i\leq g}$, then $\Psi_\Lambda=-\sum_{i=1}^{g}\lambda_i^{\ast}\otimes \lambda_i$ as an element of $\Lambda^{\ast}\otimes \Lambda$. Then $\Psi_\Lambda(\lambda_i)=-\lambda_i$. So $\Psi_\Lambda$ as an element of $\Hom(\Lambda, \Lambda)$ is the multiplication by $-1$ map. Similarly, $\Psi_{N^*}\in \End(N^*)$ is the multiplication by $-1$ map. Thus, $\Psi_N\in\End(N)$, being the dual of $\Psi_{N^*}$, is also the multiplication by $-1$ map. 
\end{proof}

\begin{remark}\label{rem:cherniso}
    If $A$ is a complex abelian variety with dual $\widehat{A}$, the first Chern class $c_1(\mathcal{P}_A)$ of the Poincar\'e bundle can be interpreted as an isomorphism $c_1(\mathcal{P}_A): H^1(A,\bZ)^{\ast}\rightarrow H^1(\widehat{A}, \bZ)$ (cf. \cite{CAV}*{\S16.4}. The above paragraph can be reinterpreted in a similar way. Indeed for a tropical real torus with integral structure $X=N_{\bR}/\Lambda$, one has $H^{1,0}(X)^{\ast}=N= H^{0,1}(\widehat{X})$ and $H^{0,1}(X)^{\ast}=\Lambda=H^{1,0}(\widehat{X})$. Thus, $c_1(\mathcal{P}_X)\in \operatorname{Aut}(\Lambda)\oplus\operatorname{Aut}(N)$ defines a natural isomorphism:
    \begin{equation}\label{eq:P-iso}
        c_1(\mathcal{P}_X): (H^{1,0}(X)\oplus H^{0,1}(X))^{\ast}\xrightarrow{\sim} H^{0,1}(\widehat{X})\oplus H^{1,0}(\widehat{X})\, .
    \end{equation}
\end{remark}

Let $\rank(\Lambda)=\rank(N)=g$. Given choice of bases $\{\lambda_1,\dots, \lambda_g\}$ of $\Lambda$ and $\{\eta_1,\dots, \eta_g\}$ of $N$, we can write down $c_1(\mathcal{P}_X)\in (N^*\otimes N)\oplus (\Lambda^*\otimes\Lambda)$ explicitly in terms of the bases.

\begin{corollary}\label{cor:c1desc}
    As an element of $(N^{\ast}\otimes N)\oplus (\Lambda^{\ast}\otimes \Lambda)$, we have:
    \[c_1(\mathcal{P}_X)= -\sum_{i=1}^{g}\eta_i^{\ast}\otimes \eta_i-\sum_{i=1}^{g}\lambda_i^{\ast}\otimes \lambda_i\, .\]
\end{corollary}

\begin{proof}
    As an element of $(N^{\ast}\otimes N)\oplus (\Lambda^{\ast}\otimes \Lambda)$, we have $c_1(\mathcal{P}_X)= \sum d_{ij}\eta^{\ast}_i\otimes \eta^{\ast}_j+\sum c_{ij}\lambda_i^{\ast}\otimes \lambda_j$ for some $c_{ij},\ d_{ij}\in \bZ$ for all $1\leq i,j\leq g$. Then, for any $1\leq k\leq g$, applying $c_1(\mathcal{P}_X)=\Psi_N+\Psi_\Lambda: N\oplus\Lambda\rightarrow N\oplus\Lambda$ to $ 0\oplus\lambda_k \in \Lambda\subset N\oplus\Lambda$ we obtain the following by, using Lemma~\ref{lem:map}:
    \begin{align*}-\lambda_k=\Psi_{\Lambda}(\lambda_k)=c_1(\mathcal{P}_X)(0\oplus \lambda_k)=\sum d_{ij}(\eta^{\ast}_i\otimes \eta_j)(0)+\sum c_{ij}(\lambda_i^{\ast}\otimes \lambda_j)(\lambda_k)\\ = \sum d_{ij}\eta^{\ast}_i(0)\eta_j+\sum c_{ij}\lambda_i^{\ast}(\lambda_k)\lambda_j=\sum_{1\leq j\leq g} c_{kj}\lambda_j\, .\end{align*}
    So $c_{kj}=-\delta_{kj}$ for all $1\leq k, j\leq g$. Similarly, applying $c_1(\mathcal{P}_X)$ to $\eta_k\oplus 0\in N\oplus\Lambda$, we see that $d_{kj}=-\delta_{kj}$ for all $1\leq k,j\leq g$.
\end{proof}

\section{A tropical Fourier--Mukai transform}\label{sec:tropFM}

As before, let $X=N_{\mathbb{R}}/\Lambda$ be a real torus  with integral structure of dimension $g$ and $\widehat{X}=\Lambda^{\ast}_{\mathbb{R}}/N^{\ast}$ be its dual. Let $p_1:X\times \widehat{X}\rightarrow X$ and $p_2:X\times\widehat{X}\rightarrow \widehat{X}$ be the two projection morphisms. Both $p_1$ and $p_2$ are proper morphisms of rational polyhedral spaces, since $X$ and $\widehat{X}$ are compact. In particular, $p_i$ induce pullback and proper pushforwards (via Poincar\`e duality, \cite{GSH}*{Theorem D}) on tropical cohomology (see \cite{GSH}*{\S4.3 and \S4.4}). Let $p_1^{\ast}: H^{p,q}(X)\rightarrow H^{p,q}(X\times \widehat{X})$ be the pullback along the projection $p_1$ and $p_{2\ast}: H^{p,q}(X\times\widehat{X})\rightarrow H^{p,q}(\widehat{X})$ be the proper pushforward along the (proper) projection $p_2$. We will next define the (cohomological) \textit{tropical Fourier--Mukai kernel} as an element of the rational tropical cohomology $H^{\bullet, \bullet}(X\times \widehat{X})_{\mathbb{Q}}\coloneqq \bigoplus_{p,q\geq 0}H^{p,q}(X\times \widehat{X})_{\mathbb{Q}}$, where $H^{p,q}(X\times\widehat{X})_{\mathbb{Q}}\coloneqq H^{p,q}(X\times\widehat{X})\otimes_{\bZ}\mathbb{Q}$.

\begin{definition}\label{Def:cohomtropFMker}
The (cohomological) \textit{tropical Fourier--Mukai kernel} is the class: 
\begin{equation}\label{eq:FMkernel}
    e^{c_1(\mathcal{P}_X)}\coloneqq\sum_{k\geq 0}\frac{1}{k!}c_1(\mathcal{P}_X)^{\wedge k}\in H^{\bullet, \bullet}(X\times \widehat{X})_{\mathbb{Q}} .
\end{equation}
\end{definition}

\begin{definition}\label{def:FMtrans}
 The \textit{Fourier--Mukai transform} on rational tropical cohomology is the following group homomorphism:
 \begin{equation}
    F=F_{E_X}: H^{\bullet, \bullet}(X)_{\mathbb{Q}}\rightarrow H^{\bullet, \bullet}(\widehat{X})_{\mathbb{Q}}, \ \ F_{E_X}(\alpha)=p_{2\ast}(e^{c_1(\mathcal{P}_X)}\cdot p_1^{\ast}\alpha)\, .
\end{equation}
\end{definition}

\subsection{Explicit description of the Fourier--Mukai transform}\label{subsec:FMdesc}

One can explicitly describe the homomorphism $F_{E_X}$ defined in Definition~\ref{def:FMtrans}.

\begin{definition}
For $p,q\geq 0$, define the group homomorphism $\alpha_{p,q}: H^{p,q}(X)\rightarrow H^{g-q,g-p}(\widehat{X})$ as the following composite map:
\begin{equation}\label{eq:alpha}
\alpha_{p,q}: H^{p,q}(X)\xrightarrow{\sim} H^{g-p,g-q}(X)^{\ast}\xrightarrow{\Psi^{\wedge g-q}_{\Lambda}\otimes (\Psi_{N})^{\wedge g-p}} H^{g-q,g-p}(\widehat{X})\, ,
\end{equation}
where the first map is the Poincar\'e duality isomorphism $H^{p,q}(X)\xrightarrow{\sim}H^{g-p,g-q}(X)^{\ast}$ induced by the perfect cup product pairing $H^{p,q}(X)\otimes H^{g-p,g-q}(X)\rightarrow H^{g,g}(X)\cong \bZ$ (see Remark~\ref{rem:poincarefourier}) and the second map is the natural map \eqref{Eqn:indhom} induced by $c_1(\mathcal{P}_X)$ on $H^{g-p,g-q}(X)^{\ast}$.  
\end{definition} 
The following result compares the restriction of the Fourier--Mukai transform $F_{E_X}$ to $H^{p,q}(X)$ with the homomorphism $\alpha_{p,q}$ defined by \eqref{eq:alpha}.

\begin{proposition}\label{prop: fourierdesc}
    $F_{E_X}|_{H^{p,q}(X)}=(-1)^{(2g-(p+q))(2g-(p+q)-1)/2} \alpha_{p,q}$.
\end{proposition}

\begin{proof}
    Fix bases $\{\lambda_1,\dots, \lambda_g\}$ of $\Lambda$ and $\{\eta_1, \dots, \eta_g\}$ of $N$ respectively and let $\{\lambda^{\ast}_1,\dots, \lambda^{\ast}_g\}$ and $\{\eta^{\ast}_1, \dots, \eta^{\ast}_g\}$ be the corresponding dual bases of $\Lambda^{\ast}$ and $N^{\ast}$. We now explicitly describe the pushforward map $p_{2\ast}:H^{\bullet, \bullet}(X\times\dual{X})\rightarrow H^{\bullet,\bullet}(\dual{X})$, the pullback map $p^{\ast}_{1}:H^{\bullet, \bullet}(X)\rightarrow H^{\bullet, \bullet}(X\times \dual{X})$ and the map $\alpha_{p,q}: H^{p,q}(X)\rightarrow H^{g-q,q-p}(\dual{X})$ defined in \eqref{eq:alpha} in terms of the chosen bases of $\Lambda,\ N,\ \Lambda^{\ast}$ and $N^{\ast}$. We introduce some notation which shall be used in this proof and throughout the remaining paper.
 \begin{notation}\label{notation:index}
 For any multi-index $I=i_1<i_2<\dots< i_p$ in $\{1,2,\dots, g\}$, we will denote $\lambda_I\coloneqq\lambda_{i_1}\wedge\dots\wedge\lambda_{i_p}$ (respectively, for $\lambda^{\ast}_I$). We analogously define $\eta_J$ and $\eta^{\ast}_J$ for any multi-index $J$ in $\{1,2,\dots, g\}$. For any multi-index $I=i_1<i_2<\dots <i_p$ in $\{1,2,\dots, g\}$, the multi-index $I^o$ will denote the complementary ordered multi-index defined as $\{1,2,\dots, g\}\setminus I$ with its natural ordering.
\end{notation}
    
    Let $\delta: H^{\bullet, \bullet}(X)=\bigoplus_{p,q\geq 0}H^{p,q}(X)\rightarrow H^{g,g}(X)\cong \bZ$ be the natural projection.  We identify $H^{\bullet, \bullet}(X\times\widehat{X})=H^{\bullet, \bullet}(X)\otimes H^{\bullet, \bullet}(\widehat{X})$ by the K\"unneth decomposition (see \cite{GSH}). Then the pushforward map $p_{2\ast}: H^{\bullet,\bullet}(X\times\widehat{X})\rightarrow H^{\bullet, \bullet}(\widehat{X})$ along the projection $p_{2\ast}:X\times\widehat{X}\rightarrow \widehat{X}$ is defined by $p_{2\ast}(x\otimes y)=\delta(x)y$ on the pure tensors of $H^{\bullet, \bullet}(X\times\widehat{X})=H^{\bullet, \bullet}(X)\otimes H^{\bullet, \bullet}(\widehat{X})$. 
    
   Since $\widehat{X}$ is a closed tropical manifold, it has a fundamental class $[\widehat{X}]$ in $Z_g(\widehat{X})$ (see \cite{GSJ}*{Definition 3.1}, which serves as the unit of the tropical intersection product on the graded ring $Z_{\bullet}(\widehat{X})$ of tropical cycles on $\widehat{X}$. Let $\operatorname{cyc}[\widehat{X}]\in H_{g,g}(\widehat{X})$ denote the cycle class obtained by applying the tropical cycle map to the fundamental class $[\widehat{X}]$ and let $[\widehat{X}]^{\ast}\in H^{0,0}(\widehat{X})$ be the Poincar\`e dual of $\operatorname{cyc}[\widehat{X}]$. We denote $[\widehat{X}]^{\ast}$ by $1\in H^{\bullet,\bullet}(\widehat{X})$, as one can check that it is the identity of the cup product on the tropical cohomology ring. The pullback $p_1^{\ast}:H^{\bullet,\bullet}(X)\rightarrow H^{\bullet,\bullet}(X\times \widehat{X})$ is given by $p_1^{\ast}(x)=x\otimes 1$.
   Note that the pullback is compatible with the ring structure on the cohomology induced by the cup product. 
   
   We now calculate the cohomological Fourier--Mukai kernel $e^{c_1(\mathcal{P}_X)}$ using Corollary~\ref{cor:c1desc} and the description of the cup product in $H^{\bullet,\bullet}(X\times\widehat{X})=\bigwedge^{\bullet}(\Lambda\oplus N^{\ast})^{\ast}\otimes \bigwedge^{\bullet}(N\oplus \Lambda^{\ast})^{\ast}$ as the exterior product:
\begin{align*}
    &e^{-\sum_{i=1}^{g}\lambda_i^{\ast}\otimes \lambda_i-\sum_{i=1}^{g}\eta^{\ast}_i\otimes\eta_i}\\
    &=\prod_{i=1}^{g}(\sum_{\nu\geq 0}\frac{1}{\nu!}\wedge^{\nu}(-\lambda_i^{\ast}\otimes \lambda_i))\wedge \prod_{i=1}^{g}(\sum_{\nu\geq 0}\frac{1}{\nu!}\wedge^{\nu}(-\eta^{\ast}_i\otimes \eta_i))\\
    &=\prod_{i=1}^{g}(1+(-\lambda_i^{\ast}\otimes \lambda_i))\wedge \prod_{i=1}^{g}(1+(-\eta^{\ast}_i\otimes \eta_i))\\
    &=\sum_{r,s=0}^{g}\sum_{\substack{I= (i_1<\dots<i_r)\\J=(j_1<\dots<j_s)}}\prod_{u=1}^{r}(\lambda_{i_u}^{\ast}\otimes (-\lambda_{i_u}))\wedge\prod_{v=1}^{s}(\eta^{\ast}_{j_v}\otimes (-\eta_{j_v}))\\
    &= \sum_{r,s=0}^{g}\sum_{\substack{I= (i_1<\dots<i_r)\\J=(j_1<\dots<j_s)}}(-1)^{(r+s)(r+s-1)/2}(\lambda^{\ast}_I\wedge \eta^{\ast}_J)\otimes((-1)^{r+s}\lambda_I\wedge \eta_J)\, .\\& 
\end{align*}

The third equality above follows from the tensor product of a cohomology class of $X$ and a cohomology class of $\widehat{X}$ in $H^{\bullet,\bullet}(X\times\widehat{X})$ being alternating, by the K\"unneth isomorphism. Thus, for a basis element $\lambda_{I_q}^{\ast}\otimes\eta_{J_p}^\ast$ in $H^{p,q}(X)=\bigwedge^q\Lambda^{\ast}\otimes\bigwedge^pN^{\ast}$, we can calculate its Fourier--Mukai transform as:
\begin{align*}
    &F_{E_X}(\lambda_{I_q}^{\ast}\otimes\eta_{J_p}^\ast)=p_{2\ast}(e^{c_1(\mathcal{P}_X)}\smile p_1^{\ast}(\lambda_{I_q}^{\ast}\otimes\eta_{J_p}^\ast))\\
    &= \sum_{r,s=0}^{g} \sum_{\substack{I= (i_1<\dots<i_r)\\J=(j_1<\dots<j_s)}} \!\!\!\!\!\!\!\!(-1)^{(r+s)(r+s+1)/2}p_{2\ast}(((\lambda^{\ast}_I\wedge \eta^{\ast}_J)\otimes(\lambda_I\wedge \eta_J))\smile ((\lambda_{I_q}^{\ast}\otimes\eta_{J_p}^\ast)\otimes 1_{\widehat{X}}))\\
    &= \sum_{r,s=0}^{g}\sum_{\substack{I= (i_1<\dots<i_r)\\J=(j_1<\dots<j_s)}}(-1)^{(r+s)(r+s+1)/2}p_{2\ast}(((\lambda^{\ast}_I\wedge \lambda_{I_q}^{\ast})\otimes ( \eta_{J}^\ast\wedge\eta^{\ast}_{J_p}))\otimes(\lambda_I\otimes \eta_J))\\
    &= \sum_{r,s=0}^{g}\sum_{\substack{I= (i_1<\dots<i_r)\\J=(j_1<\dots<j_s)}}(-1)^{(r+s)(r+s-1)/2}\delta((\lambda^{\ast}_I\wedge\lambda_{I_q}^{\ast})\otimes(\eta_{J}^\ast\wedge\eta^{\ast}_{J_p}))\ (-1)^{r+s}\lambda_I\otimes \eta_J\\
    &=(-1)^{(2g-(p+q))(2g-(p+q)-1)/2}\delta((\lambda^{\ast}_{I^o_q}\wedge\lambda_{I_q}^{\ast})\otimes(\eta_{J^o_p}^\ast\wedge\eta^{\ast}_{J_p})) \ (-1)^{p+q}\lambda_{I^o_q}\otimes \eta_{J^o_p}\\
    &= (-1)^{(2g-(p+q))(2g-(p+q)-1)/2} \beta_{p,q}(\lambda_{I_q}^{\ast}\otimes\eta_{J_p}^\ast)\, ,
\end{align*}
where $\beta_{p,q}(\lambda_{I_q}^{\ast}\otimes\eta_{J_p}^\ast)\coloneqq \delta((\lambda^{\ast}_{I^o_q}\wedge\lambda_{I_q}^{\ast})\otimes(\eta_{J^o_p}^\ast\wedge\eta^{\ast}_{J_p})) \ (-1)^{p+q}\lambda_{I^o_q}\otimes \eta_{J^o_p}$. The second equality above follows from the definition of intersection product in $H^{\bullet, \bullet}(X\times\widehat{X})$.
   
We now describe $\alpha_{p,q}$ in terms of the chosen bases of $\Lambda,\ N,\ \Lambda^{\ast}$ and $N^{\ast}$. Using Notation~\ref{notation:index} and description \eqref{Eqn:Toruscohom} of the tropical cohomology of real tori with integral structures, the Poincar\'e duality isomorphism $H^{p,q}(X)\xrightarrow{\sim} H^{g-p, g-q}(X)^{\ast}$ is the map $\lambda^{\ast}_{I_q}\otimes \eta^{\ast}_{J_p}\mapsto \delta((\lambda^{\ast}_{I^o_q}\wedge\lambda^{\ast}_{I_q})\otimes (\eta^{\ast}_{J^o_p}\wedge\eta^{\ast}_{J_p}))\lambda_{I_q^o}\otimes \eta_{J_p^o}$. By \eqref{eq:alpha}, we then have:
    \begin{equation}\label{Eq:alphadesc}
        \alpha_{p,q}(\lambda^{\ast}_{I_q}\otimes\eta^{\ast}_{J_p})\coloneqq \delta((\lambda^{\ast}_{I^o_q}\wedge\lambda^{\ast}_{I_q})\otimes (\eta^{\ast}_{J^o_p}\wedge \eta^{\ast}_{J_p}))\ \Psi^{\wedge g-q}_\Lambda(\lambda_{I_q^o})\otimes \Psi^{\wedge g-p}_N(\eta_{J_p^o})\, .
    \end{equation}
 By applying Lemma~\ref{lem:map} to \eqref{Eq:alphadesc}, we finally obtain: 
 $$\alpha_{p,q}(\lambda^{\ast}_{I_q}\otimes\eta^{\ast}_{J_p})=(-1)^{p+q}\delta((\lambda^{\ast}_{I^o_q}\wedge\lambda^{\ast}_{I_q})\otimes (\eta^{\ast}_{J^o_p}\wedge \eta^{\ast}_{J_p}))\lambda_{I_q^o}\otimes\eta_{J_p^o}=\beta_{p,q}(\lambda^{\ast}_{I_q}\otimes\eta^{\ast}_{J_p})\, .$$
 Thus, plugging in $\alpha_{p,q}(\lambda^{\ast}_{I_q}\otimes\eta^{\ast}_{J_p})=\beta_{p,q}(\lambda^{\ast}_{I_q}\otimes\eta^{\ast}_{J_p})$ in the description of $F_{E_X}(\lambda_{I_q}^{\ast}\otimes\eta_{J_p}^\ast)$ obtained above, we obtain the desired identity on basis elements of $H^{p,q}(X)$ and thereby on entire $H^{p,q}(X)$ since $F_{E_X}$ is a homomorphism.
 
\end{proof}

\begin{remark}\label{rem:poincarefourier}
    In \cite{JSS} and \cite{GSH} the authors prove Poincar\'e duality for general tropical manifolds (using Dolbeault cohomology of tropical superforms, and sheaf theoretic methods, respectively). Proposition~\ref{prop: fourierdesc} demonstrates that the Fourier--Mukai transform provides Poincar\'e duality (up to an isomorphism $H^{i,j}(X)^{\ast}\xrightarrow{\sim}H^{j,i}(\widehat{X})$) for tropical cohomology of tropical abelian varieties. As a corollary, we also see that the Fourier--Mukai transform $F_{E_X}:H^{\bullet, \bullet}(X)\rightarrow H^{\bullet, \bullet}(\dual{X})$ is compatible with the bigrading of the tropical cohomology rings.
\end{remark}

\begin{corollary}\label{cor:FMpq}
    $F_{E_X}|_{H^{p,q}(X)}: H^{p,q}(X)\xrightarrow{\sim} H^{g-q,g-p}(\widehat{X})$ is an isomorphism of abelian groups.
\end{corollary}

\begin{proof}
    This follows from Proposition~\ref{prop: fourierdesc} since $\alpha_{p,q}: H^{p,q}(X)\rightarrow H^{g-q, g-p}(\widehat{X})$ is an isomorphism.
\end{proof}

\begin{remark}
    \textit{A priori}, it is not clear from Definition~\ref{def:FMtrans} that the Fourier--Mukai transform $F_{E_X}$ restricts to a homomorphism between the integral tropical cohomology $H^{\bullet,\bullet}(X)\rightarrow H^{\bullet,\bullet}(\widehat{X})$. But this follows from Corollary~\ref{cor:FMpq}.
\end{remark}

\section{A generalized tropical Poincar\'e formula}\label{sec:GenPoincareformula}

The goal of this section is to prove the following generalized tropical Poincar\'e formula for an ample line bundle on a tropical abelian variety. 

\begin{theorem}\label{propgenpoin}
Let $[D]\in \operatorname{CaCl}(X)$ be an ample tropical Cartier divisor on $X=N_\bR/\Lambda$ with $d=h^0(X,\sL(D))$. Let $c_{[D]}=[D]^{\cdot g-1}/(d(g-1)!)$, where $[D]^{\cdot g-1}$ is the $(g-1)$-fold tropical intersection product of $D$. Then, for $0\leq p\leq g$, the following equality holds in tropical homology:
\begin{equation}\label{eqgenPoin}
\frac{[D]^{\cdot p}}{p!}=d\frac{c_{[D]}^{\star \ g-p}}{(g-p)!} \, ,\end{equation}
\end{theorem}
 
\subsection{Fourier--Mukai action on intersection powers of nondegenerate line bundles}\label{subsec:MainProp}
We first describe the action of the cohomological tropical Fourier--Mukai transform $F_{E_X}:H^{\bullet,\bullet}(X)\rightarrow H^{\bullet,\bullet}(\widehat{X})$ on the first Chern class of intersection powers of nondegenerate tropical line bundles on real tori with integral structures. This is the main technical result that will be useful in proving Theorem~\ref{propgenpoin}. Recall the natural homomorphism of tropical real tori with integral structures $\phi_{\sL}: X\rightarrow \widehat{X}$, defined in \S\ref{subsec:dualrealtori}.

\begin{theorem}\label{mainthm}
    Let $\sL$ be a nondegenerate line bundle on a real torus with integral structure $X=N_\bR/\Lambda$ of dimension $g$. Let the first Chern class $c_1(\sL)$ of $\sL$ be $E\in H^{1,1}(X)=(\Lambda\otimes N)^{\ast}$. Let $F_{E_X}:H^{\bullet, \bullet}(X)\rightarrow H^{\bullet, \bullet}(\widehat{X})$ be the Fourier--Mukai transform on tropical cohomology. Then the following equality holds in $H^{\bullet, \bullet}(\widehat{X})$, for any $0\leq p\leq g$:
    \[ F_{E_X}\left(\frac{c_1(\sL^{\cdot p})}{p!}\right)=\frac{(-1)^{g-p}}{\det E}\phi_{\sL\ast}\left(\frac{c_1(\sL^{\cdot g-p})}{(g-p)!}\right)\,.\]
\end{theorem}

\begin{proof}
    Let the first Chern class $c_1(\sL)$ of the nondegenerate line bundle $\sL\in \pic(X)$ be $E\in H^{1,1}(X)=\Lambda^{\ast}\otimes N^{\ast}$. Identifying $\Lambda^{\ast}\otimes N^\ast$ with $\Hom(\Lambda, N^\ast)$, we may view $E$ as a homomorphism $\Lambda\rightarrow N^\ast$ of free $\bZ$-modules. We may identify $E$ with its corresponding $g\times g$ matrix with integer entries, denoted by $[E]$, which is symmetric by Theorem~\ref{thm:tropical-appell-humbert}. Consequently, $E$ determines a symmetric bilinear form on $N_\bR$ such that $E(\Lambda\times N)\subseteq \bZ$.

    Let $\{\lambda_1,\dots, \lambda_g\}$ be a basis for $\Lambda$ and $\{\eta_1,\dots, \eta_g\}$ be a basis for $N$ with the corresponding dual basis of $\Lambda^{\ast}$ being $\{\lambda_1^{\ast}, \dots, \lambda_g^{\ast}\}$ and of $N^\ast$ being $\{\eta_1^{\ast}, \dots, \eta_g^{\ast}\}$. Letting $E_{i,j}\coloneqq E(\lambda_i, \eta_j)$ for all $1\leq i,j\leq g$, we can write $c_1(\sL)=E=\sum_{1\leq i,j\leq g}E_{i,j}\lambda_i^{\ast}\otimes\eta_j^{\ast}\in\Lambda^{\ast}\otimes N^\ast=H^{1,1}(X)$. Let the matrix of $E$ with respect to the above chosen bases of $\Lambda$ and $N^{\ast}$ be $[E]\coloneqq(E_{i,j})_{1\leq i,j\leq g}$. By definition of the intersection product of line bundles:
    \begin{align*}
        & c_1(\sL^{\cdot p})=c_1(\sL)^{\wedge p}= \bigwedge_{k=1}^{p}\left( \sum_{1\leq i_k,j_k\leq g}E_{i_k,j_k}\lambda_{i_k}^{\ast}\otimes\eta_{j_k}^{\ast}\right)\\&= \sum_{I_p, J_p\subseteq [g]}\sum_{\sigma\in \mathfrak{S}_p}\sum_{\tau\in \mathfrak{S}_p}\operatorname{sgn}(\tau)\operatorname{sgn}(\sigma)\prod_{k=1}^{p}E_{i_{\sigma(k)},j_{\tau(k)}}\lambda^{\ast}_{I_p}\otimes \eta^{\ast}_{J_p}\, ,
    \end{align*}
    where the first sum runs over ordered size $p$ subsets $I_p: i_1<i_2<\dots i_p$ and $J_p: j_1<j_2<\dots <j_p$ of $[g]\coloneqq\{1,2,\dots, g\}$. The inner two sums run over $\sigma,\tau\in \operatorname{Aut}\{1,2,\dots, p\}\cong \mathfrak{S}_p$. We can further rewrite the summand above as: 
    \[\operatorname{sgn}(\tau\circ \sigma^{-1})\prod_{k=1}^{p}E_{i_{k},j_{\tau(\sigma^{-1}(k))}}\lambda^{\ast}_{I_p}\otimes \eta^{\ast}_{J_p}\,.\]
    Since $\mathfrak{S}_p$ acts transitively on itself via right multiplication $(\sigma, \tau)\mapsto \tau\circ\sigma^{-1}$, we obtain:
    \[\frac{c_1(\sL^{\cdot p})}{p!}=\sum_{I_p, J_p\subseteq [g]}\sum_{\tau\in \mathfrak{S}_p}\operatorname{sgn}(\tau)\prod_{k=1}^{p}E_{i_{k},j_{\tau(k)}}\lambda^{\ast}_{I_p}\otimes \eta^{\ast}_{J_p}\,. \]
    For any pair of subsets of $A, B\subseteq [g]$, let $[E]_{A,B}$ denote the sub-matrix of $[E]$ formed by the rows indexed by $A$ and columns indexed by $B$. Then, by the Leibniz formula for determinants, we have:
    \begin{equation}\label{detp}
        \frac{c_1(\sL^{\cdot p})}{p!}=\sum_{I_p, J_p\subseteq [g]}\det[E]_{I_p,J_p}\lambda^{\ast}_{I_p}\otimes \eta^{\ast}_{J_p}\, .
    \end{equation}
    Applying the Fourier--Mukai transform to \eqref{detp}, using Proposition~\ref{prop: fourierdesc}, we obtain:
    \begin{align*}\label{detp2}
        &F_{E_X}\left(\frac{c_1(\sL^{\cdot p})}{p!}\right)=\sum_{I_p, J_p\subseteq [g]}\det[E]_{I_p,J_p}F_{E_X}(\lambda^{\ast}_{I_p}\otimes \eta^{\ast}_{J_p})\\
        &=(-1)^{g-p}\sum_{I_p, J_p\subseteq [g]}\det[E]_{I_p,J_p}\alpha_{p,p}(\lambda^{\ast}_{I_p}\otimes \eta^{\ast}_{J_p})\,.
    \end{align*}
    By the definition of the homomorphism $\alpha_{p,p}:H^{p,p}(X)\rightarrow H^{g-p,g-p}(\widehat{X})$ in \eqref{eq:alpha}, we have:
    \begin{equation}\label{eq:LHS}
        F_{E_X}\!\!\left(\frac{c_1(\sL^{\cdot p})}{p!}\right)=(-1)^{g-p}\!\!\!\!\sum_{I_p, J_p\subseteq [g]}\!\!\!\det[E]_{I_p,J_p}\delta((\lambda^{\ast}_{I_p^o}\wedge \lambda^{\ast}_{I_p})\otimes (\eta^{\ast}_{J_p^o}\wedge\eta^{\ast}_{J_p}))\lambda_{I_p^o}\otimes \eta_{J_p^o}\,.
    \end{equation}
    Next, we compute the class of $\phi_{\sL\ast}\left(c_1(\sL^{\cdot g-p})/(g-p)!\right)$. By a similar computation as for \eqref{detp}, we obtain:
    \begin{equation}\label{detg-p}
        \frac{c_1(\sL^{\cdot g-p})}{(g-p)!}=\sum_{I_{g-p}, J_{g-p}\subseteq [g]}\det[E]_{I_{g-p},J_{g-p}}\lambda^{\ast}_{I_{g-p}}\otimes \eta^{\ast}_{J_{g-p}}\,.
    \end{equation}
    Since $X$ and its dual $\widehat{X}$ are compact, the pushforward map $\phi_{\sL\ast}:H^{\ast,\ast}(X)\rightarrow H^{\ast,\ast}(\widehat{X})$ induced by the map $\phi_{\sL}:X\rightarrow \widehat{X}$ is the composition:
    \begin{equation}\label{eqn:pushforward}
    H^{p,q}(X)\xrightarrow{\sim} H_{g-p,g-q}(X)\xrightarrow{\widetilde{\phi}_{\sL\ast}}H_{g-p,g-q}(\widehat{X})\xrightarrow{\sim}H^{p,q}(\widehat{X})\,,
    \end{equation}
    where the first and third maps are the Poincar\'e duality isomorphisms, and the map $\widetilde{\phi}_{\sL\ast}$ is the pushforward map on homology. By Lemma~\ref{lem:analyticrep}, the map $\phi_{\sL}: X\rightarrow \widehat{X}$ is induced by the homomorphism $\phi_{E}:N_\bR\rightarrow \Lambda^{\ast}_\bR$ on the universal cover, defined by $\phi_E(x)=-E(-,x)$. Thus, $\phi_{\sL}$ naturally induces the map $\phi_E:N\rightarrow \Lambda^{\ast}$, given by the matrix $-E$. The fact that the map $\phi_E: N_\bR\rightarrow \Lambda^{\ast}_\bR$ descends to $\phi_{\sL}:X\rightarrow \widehat{X}$ is equivalent to $\phi_E$ inducing a homomorphism $\Lambda\rightarrow N^{\ast}$ given by the matrix $-E^*=-E$ (since $E$ is symmetric). In particular, the induced maps $N\xrightarrow{-E} \Lambda^{\ast}$ and $\Lambda\xrightarrow{-E} N^{\ast}$ are the pushforward maps on homology $\phi_{\sL\ast}: H_{0,1}(X)\rightarrow H_{0,1}(\widehat{X})$ and $\phi_{\sL\ast}: H_{1,0}(X)\rightarrow H_{1,0}(\widehat{X})$ respectively. Since $H_{g-p,g-q}(X)=\bigwedge^{g-q}\Lambda\otimes \bigwedge^{g-p}N$ and $H_{g-p,g-q}(\widehat{X})=\bigwedge^{g-q}N^{\ast}\otimes\bigwedge^{g-p}\Lambda^\ast$, it follows that the pushforward $\widetilde{\phi}_{\sL\ast}: H_{g-p, g-q}(X)\rightarrow H_{g-p,g-q}(\widehat{X})$ is given by:
    \[\widetilde{\phi}_{\sL\ast}: \bigwedge^{g-q} \phi_E\otimes \bigwedge^{g-p}\phi_E: H_{g-p,g-q}(X)\rightarrow H_{g-p,g-q}(\widehat{X})\,. \]
    By \eqref{detg-p}, the class $c_1(\sL^{\cdot g-p}/(g-p)!)$ belongs in $H^{g-p, g-p}(X)$, so we will only focus on the pushforward $\phi_{\sL\ast}$ restricted to $H^{g-p, g-p}(X)$. The Poincar\'e dual of a basis element $\lambda^{\ast}_{I_{g-p}}\otimes \eta^{\ast}_{J_{g-p}}\in H^{g-p,g-p}(X)$ is $\delta((\lambda^{\ast}_{I^o_{g-p}}\wedge \lambda^{\ast}_{I_{g-p}})\otimes(\eta^{\ast}_{J^o_{g-p}}\wedge \eta^{\ast}_{J_{g-p}}))\lambda_{I^o_{g-p}}\otimes \eta_{J^o_{g-p}}\in H_{p,p}(X)$. Thus, the image of $\lambda^{\ast}_{I_{g-p}}\otimes \eta^{\ast}_{J_{g-p}}$ under the composition of the first two maps of \eqref{eqn:pushforward} is:
    \begin{equation}\label{eqn:im2map}
        \lambda^{\ast}_{I_{g-p}}\otimes \eta^{\ast}_{J_{g-p}}\longmapsto \delta((\lambda^{\ast}_{I^o_{g-p}}\wedge \lambda^{\ast}_{I_{g-p}})\otimes(\eta^{\ast}_{J^o_{g-p}}\wedge \eta^{\ast}_{J_{g-p}}))\!\!\bigwedge_{i\in I^o_{g-p}}\!\!(-E\lambda_i)\otimes\!\! \bigwedge_{j\in J^o_{g-p}}\!\!(-E\eta_j)\,.
    \end{equation}
    Now note that $E\lambda_i=\sum_{j=1}^{g}E_{i,j}\eta^{\ast}_j$ and similarly $E\eta_j= \sum_{i=1}^{g}E_{i,j}\lambda^{\ast}_i$. Thus, $\bigwedge_{i_k\in I^o_{g-p}}E\lambda_{i_k}$ equals:
    \begin{equation}\label{eqnwedge1}
        \bigwedge_{i_k\in I^o_{g-p}}(\sum_{j_k=1}^{g}E_{i_k,j_k}\eta^{\ast}_{j_k})=\sum_{\mathcal{J}_p\subseteq[g]}\sum_{\sigma\in\mathfrak{S}_p}\operatorname{sgn}(\sigma)\!\!\prod_{\substack{k=1\\ i_k\in I^o_{g-p} }}^{p}E_{i_k,j_{\sigma(k)}}\eta^{\ast}_{\mathcal{J}_p}=\sum_{\mathcal{J}_p\subseteq[g]}\det[E]_{I^o_{g-p}, \mathcal{J}_p}\eta^{\ast}_{\mathcal{J}_p}\,,
    \end{equation}
    where the outer sum runs over size $p$ ordered subsets $\mathcal{J}_p: j_1<j_2<\dots<j_p$ of $[g]$, and the last equality follows from Leibniz formula for determinants. Similarly, we obtain:
    \begin{equation}\label{eqnwedge2}
        \bigwedge_{j_k\in J^o_{g-p}}E\eta_{j_k}=\sum_{\mathcal{I}_p\subseteq[g]}\det[E]_{\mathcal{I}_p, J^o_{g-p}}\lambda^{\ast}_{\mathcal{I}_p}\,,
    \end{equation}
    where the sum runs over size $p$ ordered subsets $\mathcal{I}_p: i_1<i_2<\dots<i_p$ of $[g]$. Thus, the image in \eqref{eqn:im2map} is the following element of $H_{g-p,g-p}(\widehat{X})$:
    \begin{equation}\label{eqnwedge3}
         \delta((\lambda^{\ast}_{I^o_{g-p}}\wedge \lambda^{\ast}_{I_{g-p}})\otimes(\eta^{\ast}_{J^o_{g-p}}\wedge \eta^{\ast}_{J_{g-p}}))\sum_{\mathcal{I}_p, \mathcal{J}_p\subseteq[g]}\det[E]_{I^o_{g-p}, \mathcal{J}_p}\det[E]_{\mathcal{I}_p, J^o_{g-p}}\eta^{\ast}_{\mathcal{J}_p}\otimes \lambda^{\ast}_{\mathcal{I}_p}\,.
    \end{equation}
    For brevity, let $\delta_{I,J}\coloneqq \delta((\lambda^{\ast}_{I^o}\wedge \lambda^{\ast}_{I})\otimes(\eta^{\ast}_{J^o}\wedge \eta^{\ast}_{J}))\in\bZ$ for any pair of ordered subsets $I,J\subseteq [g]$. Then taking the Poincar\'e dual of \eqref{eqnwedge3}, we have:
    \begin{equation}\label{eqnwedge4}
        \phi_{\sL\ast}( \lambda^{\ast}_{I_{g-p}}\otimes \eta^{\ast}_{J_{g-p}})=\delta_{I_{g-p}, J_{g-p}}\!\!\!\!\sum_{\mathcal{I}_p, \mathcal{J}_p\subseteq[g]}\!\!\!\!\!\det[E]_{I^o_{g-p}, \mathcal{J}_p}\det[E]_{\mathcal{I}_p, J^o_{g-p}}\delta_{\mathcal{I}_p, \mathcal{J}_p} \ \lambda_{\mathcal{I}^o_p}\!\otimes\!\eta_{\mathcal{J}^o_p}\,.
    \end{equation}

    We apply the pushforward on cohomology $\phi_{\sL\ast}$ to \eqref{detg-p}, using \eqref{eqnwedge4}, to obtain:
    \begin{equation}\label{EqpushA}
    \begin{split}
       &\phi_{\sL\ast}\left(\frac{c_1(\sL^{\cdot g-p})}{(g-p)!}\right)\\&=
       \sum_{\substack{\mathcal{I}_p\subseteq [g]\\ \mathcal{J}_p\subseteq[g]}}\sum_{\substack{I_{g-p}\subseteq [g]\\ J_{g-p}\subseteq [g]}} \!\!\!\!\! \delta_{I_{g-p}, J_{g-p}}\delta_{\mathcal{I}_p, \mathcal{J}_p}\det[E]_{\mathcal{I}_p, J^o_{g-p}}\!\det[E]_{I_{g-p},J_{g-p}}\det[E]_{I^o_{g-p}, \mathcal{J}_p}\lambda_{\mathcal{I}^o_p}\otimes\eta_{\mathcal{J}^o_p}\, .
    \end{split}
    \end{equation}
    By Jacobi's identity \cite[Equation~(0.8.4.1)]{Horn} for the non-singular matrix $[E]$, we have:
    \begin{equation}\label{EqnJacobi}
    \det[E]_{I_{g-p}, J_{g-p}}=\epsilon(I_{g-p}, J_{g-p})\det[E]\det[E]^{-1}_{J^o_{g-p}, I^o_{g-p}}\,,
    \end{equation}
    where $\epsilon_{I_{g-p}, J_{g-p}}=(-1)^{\sum I_{g-p}+\sum J_{g-p}}$, using the notation $\sum I_{g-p}\coloneqq \sum_{i\in I_{g-p}}i$ (and similarly $\sum J_{g-p}$). Now for any subset $K\subseteq [g]$ such that $k=|K|$, let $\omega_K\in \mathfrak{S}_g$ be the unique permutation for which $\omega_K(1), \omega_K(2),\dots, \omega_K(k)$ are the elements of $K$ in increasing order and $\omega_K(k+1),\dots, \omega_K(g)$ are the elements of $K^o\coloneqq[g]\setminus K$ in increasing order. Then note that $\lambda^{\ast}_K\otimes \lambda^{\ast}_{K^o}=\operatorname{sgn}(\omega_K)\lambda^{\ast}_{[g]}$. One can check that $\operatorname{sgn}(\omega_K)=(-1)^{\sum K-(1+2+\dots+|K|)}$.
    Since, $\lambda^{\ast}_{I^o_{g-p}}\wedge \lambda^{\ast}_{I_{g-p}}=\operatorname{sgn}(\omega_{I^o_{g-p}})\lambda^{\ast}_{[g]}$ and, similarly,  $\eta^{\ast}_{J^o_{g-p}}\wedge \eta^{\ast}_{J_{g-p}}=\operatorname{sgn}(\omega_{J^o_{g-p}})\eta^{\ast}_{[g]}$, we have:
    \[\delta_{I_{g-p}, J_{g-p}}=\operatorname{sgn}(\omega_{I^o_{g-p}})\operatorname{sgn}(\omega_{J^o_{g-p}})\delta(\lambda^{\ast}_{[g]}\otimes \eta^{\ast}_{[g]})=(-1)^{\sum I^o_{g-p}+\sum J^o_{g-p}+2(1+2+\dots+p)}\,.\]
    Therefore, 
    \begin{equation}\label{Eqn:justnow}
        \epsilon(I_{g-p}, J_{g-p})\delta_{I_{g-p}, J_{g-p}}=(-1)^{\sum I_{g-p}+\sum J_{g-p}+\sum I^o_{g-p}+\sum J^o_{g-p}}=(-1)^{2\sum [g]}=1\,.
    \end{equation} Putting together \eqref{EqpushA}, \eqref{EqnJacobi}, and \eqref{Eqn:justnow}, we get: 
    \begin{align*}
        &\phi_{\sL\ast}\left(\frac{c_1(\sL^{\cdot g-p})}{(g-p)!}\right)\\&=\det[E] \!\! \sum_{\mathcal{I}_p, \mathcal{J}_p\subseteq[g]}\delta_{\mathcal{I}_p, \mathcal{J}_p} \!\!\!\! \sum_{I_{g-p}, J_{g-p}\subseteq [g]} \!\!\!\! \det[E]_{\mathcal{I}_p, J^o_{g-p}}\det[E]^{-1}_{J^o_{g-p},I^o_{g-p}}\det[E]_{I^o_{g-p}, \mathcal{J}_p}\ \lambda_{\mathcal{I}^o_p}\otimes\eta_{\mathcal{J}^o_p}\\&= \det[E]\sum_{\mathcal{I}_p, \mathcal{J}_p\subseteq[g]}\delta_{\mathcal{I}_p, \mathcal{J}_p}\sum_{I_{g-p}\subseteq [g]}\det[I_g]_{\mathcal{I}_p,I^o_{g-p}}\det[E]_{I^o_{g-p}, \mathcal{J}_p}\ \lambda_{\mathcal{I}^o_p}\otimes\eta_{\mathcal{J}^o_p}\\&= \det[E]\sum_{\mathcal{I}_p, \mathcal{J}_p\subseteq[g]}\delta_{\mathcal{I}_p, \mathcal{J}_p}\det[E]_{\mathcal{I}_p, \mathcal{J}_p}\ \lambda_{\mathcal{I}^o_p}\otimes\eta_{\mathcal{J}^o_p}\\
        &= (-1)^{g-p}\det[E] F_{E_X}\left(\frac{c_1(\sL^{\cdot p})}{p!}\right)\,.
    \end{align*}
    The second equality above follows from $\sum_{J_{g-p}\subseteq[g]}\det[E]_{\mathcal{I}_p, J^o_{g-p}}\det[E]^{-1}_{J^o_{g-p},I^o_{g-p}}=\det[I_g]_{\mathcal{I}_p,I^o_{g-p}}$, obtained via Cauchy--Binet formula for minors \cite[\S 0.8.7]{Horn}. A similar application of the Cauchy--Binet formula gives the third equality above as well. The last equality follows from \eqref{eq:LHS}. If $\sL\in \pic(X)$ is a nondegenerate line bundle, then $\det[E]\neq 0$, where $E=c_1(\sL)$. Thus, we obtain the desired identity $F_{E_X}\left(\frac{c_1(\sL^{\cdot p})}{p!}\right)= \frac{(-1)^{g-p}}{\det E}\phi_{\sL\ast}\left(\frac{c_1(\sL^{\cdot g-p})}{(g-p)!}\right)$.
\end{proof}

\subsection{Analytic Riemann--Roch for ample tropical line bundles}\label{subsec:AlgRR}In this section we note a tropical analogue of the analytic Riemann--Roch theorem for ample line bundles on abelian varieties. Some of these results are implicitly present in Sumi's work \cite{KS} on Riemann--Roch inequality for tropical abelian varieties. 

The space $\Gamma(X,\sL)$ of global sections of a tropical line bundle $\sL\in \pic(X)$ forms a convex polyhedron by \cite{KS}*{Theorem 36}. Sumi (\cite[Theorem 3]{KS}) also proves that the dimension $h^0(X,\sL)$ of the polyhedron $\Gamma(X, \sL)$ agrees with the definition of $h^0$ proposed in \cite{DC}. 

\begin{lemma}\label{rrlemma}
    Let $\sL$ be an ample line bundle on $X=N_{\bR}/\Lambda$ with first Chern class $c_1(\sL)=E$. Then $h^0(X, \sL)=\det E$. 
\end{lemma}

\begin{proof}
  A line bundle $\sL\in \pic(X)$ is ample if and only if its first Chern class $E\coloneqq c_1(\sL)$ is a positive-definite bilinear form on $N_\bR$. In particular, note that $E\in \Lambda^{\ast}\otimes N^{\ast}$, whereby we can consider $E$ as a homomorphism $\rho_E: \Lambda\rightarrow N^{\ast}$. Then by \cite{KS}*{Theorem 36}, $h^0(X, \sL)=|\operatorname{coker}\rho_E|$. Since $E$ is positive-definite, one sees that $\det E=|\operatorname{coker} \rho_E|$. 
\end{proof}

As a consequence, we obtain the tropical analogue of the analytic Riemann--Roch for abelian varieties (see \cite[\S16]{MumAV}, \cite[Corollary~3.6.2]{CAV}). Recall $\deg\phi_{\sL}$ as defined in Remark~\ref{rem:isonondegen} (ii).

\begin{corollary}\label{lem:tropanRR}
    Let $\sL\in\pic(X)$ be an ample tropical line bundle on $X=N_\bR/\Lambda$ with first Chern class $c_1(\sL)=E$. Then $\deg\phi_{\sL}=\det E$.
\end{corollary}

\begin{proof}
    Let $\rho_E:\Lambda\rightarrow N^*$ be the lattice homomorphism induced by $c_1(\sL)=E$. Since $\sL$ is ample, $\rho_E$ is injective. Let $\overline{\rho}_E:N_\bR\rightarrow N_\bR^*$ be the map $x\mapsto E(x,-)$, which is an isomorphism since $E$ is positive-definite. Then $\overline{\rho}_E|_\Lambda=\rho_E$ and $\overline{\rho}_E^{-1}(N^*)=\Lambda(\sL)$ as defined in \S\ref{subsec:dualrealtori}. This is obtained by naturally identifying $N^*_\bR=\Lambda^*_\bR$ and $\overline{\rho}_E=-\Phi_E$ (as given in Lemma~\ref{lem:analyticrep}). We have the following factorization:
    \[\begin{tikzcd}
	\Lambda && {N^*} \\
	& {\Lambda(\sL)}
	\arrow["{\rho_E}", hook, from=1-1, to=1-3]
	\arrow[hook, from=1-1, to=2-2]
	\arrow["{\overline{\rho}_E}"', from=2-2, to=1-3]
\end{tikzcd}\]
Now $\overline{\rho}_E:\Lambda(\sL)\xrightarrow{\sim} N^*$ is an isomorphism since $\overline{\rho}_E:N_\bR\rightarrow N^*_\bR$ is an isomorphism. Thus, $\deg\phi_\sL\coloneqq[\Lambda(\sL):\Lambda]=|\operatorname{coker}\rho_E|=h^0(X,\sL)=\det E$ by Lemma~\ref{rrlemma}.
\end{proof}

As an immediate corollary of Theorem~\ref{mainthm} and Lemma~\ref{rrlemma}, we obtain the following tropical analogue of \cite[Proposition 5]{Beauv2} (also see \cite[Theorem 16.5.]{CAV}) at the level of cohomology.

\begin{corollary}\label{maincor}
    Let $\sL$ be an ample line bundle on a real torus with integral structure $X=N_\bR/\Lambda$  of dimension $g$ and let $d=h^0(X, \sL)$. Let $F_{E_X}:H^{\bullet, \bullet}(X)\rightarrow H^{\bullet, \bullet}(\widehat{X})$ be the Fourier--Mukai transform on tropical cohomology and let $\phi_{\sL}: X\rightarrow \widehat{X}$ be the natural morphism of tropical real tori with integral structures as defined earlier. Then the following equality holds in $H^{\bullet, \bullet}(\widehat{X})$, for any $0\leq p\leq g$:
    \[ F_{E_X}\left(\frac{c_1(\sL^{\cdot p})}{p!}\right)=\frac{(-1)^{g-p}}{d}\phi_{\sL\ast}\left(\frac{c_1(\sL^{\cdot g-p})}{(g-p)!}\right)\,.\]
\end{corollary}

\subsection{Proof of Theorem~\ref{propgenpoin}} Our strategy will be to prove the Poincar\'e dual of the desired statement in tropical cohomology. To do so, we first need to understand the Poincar\'e dual of tropical Cartier divisors and the Poincar\'e dual of the homological Pontryagin product in tropical cohomology.

\begin{remark}\label{rem:divdual}
    Let $X$ be a compact tropical manifold of dimension $g$. Let $\operatorname{cyc}:\operatorname{CDiv}(X)\rightarrow H_{g-1, g-1}(X)$ be the cycle class map obtained by the natural identification $\operatorname{CDiv}(X)\cong Z_{g-1}(X)$ . By \cite{GSJ}*{Theorem 3.3}, $\operatorname{cyc}$ is defined by $\operatorname{cyc}(D)=c_1(\sL(D))\frown \operatorname{cyc}[X]$, where $[X]\in Z_g(X)$ is the fundamental class of $X$. It follows that $\operatorname{cyc}$ descends to a homomorphism, which we denote by the same notation, $\operatorname{cyc}: \operatorname{CaCl}(X)\rightarrow H_{g-1,g-1}(X)$. Here $\operatorname{CaCl}(X)$ is the class group of tropical Cartier divisors on $X$. Consequently, we see that $\operatorname{cyc}([D])\in H_{g-1,g-1}(X)$ is Poincar\'e dual to $c_1(\sL(D))\in H^{1,1}(X)$. 
\end{remark}

\begin{lemma}\label{lem:cycX}
Let $X=N_\bR/\Lambda$ be a real torus with integral structure of dimension $g$. Then there exist bases $\{\lambda_1,\dots, \lambda_g\}$ for $\Lambda$ and $\{\eta_1,\dots, \eta_g\}$ for $N$, for which $\operatorname{cyc}([X])=\lambda_{[g]}\otimes \eta_{[g]}\in H_{g,g}(X)=\bigwedge^g\Lambda\otimes\bigwedge^gN$.
\end{lemma}

\begin{proof}
Let $\{\lambda_1,\dots, \lambda_g\}$ and $\{\eta_1,\dots, \eta_g\}$ be arbitrary fixed bases for $\Lambda$ and $N$ respectively. Since $\operatorname{cyc}([X])\in H_{g,g}(X)=\bigwedge^g\Lambda\otimes\bigwedge^gN$, we know $\operatorname{cyc}([X])=c.\lambda_{[g]}\otimes \eta_{[g]}$ for some $c\in\bZ$. We know that the Poincar\'e duality is given by the map $-\frown \operatorname{cyc}([X]): H^{p,q}(X)\rightarrow H_{g-p, g-q}(X)$. Since cap product is the interior product on exterior algebra (see \cite{GSJ}*{Equation (6.3)}), by a similar argument as in \cite{GSJ}*{Lemma 9.6}, we see that the Poincar\'e dual of $\lambda^{\ast}_I\otimes\eta_J\in H^{p.q}(X)$ is $c.\lambda_{I^o}\otimes \eta_{J^o}\in H_{g-p, g-q}(X)$. Since the Poincar\'e duality map is an isomorphism, we must have $c=\pm 1$. By rescaling a basis element of $\Lambda$ by $\pm 1$, we may assume $c=1$, whereby $\operatorname{cyc}([X])=\lambda_{[g]}\otimes \eta_{[g]}\in H_{g,g}(X)$.
\end{proof}

\subsubsection{Pontryagin product on cohomology}\label{subsubsec:cohompontryagin} We now describe the binary operation on tropical cohomology which is Poincar\'e dual to the Pontryagin product on tropical homology. Under the identification $H_{\bullet, \bullet}(X)=\bigwedge^{\bullet}\Lambda\otimes \bigwedge^{\bullet}N$, the tropical Pontryagin product $\star$ is the wedge product, that is, $(\lambda_I\otimes \eta_J)\star (\lambda_{I'}\otimes \eta_{J'})=(\lambda_I\wedge \lambda_{I'})\otimes (\eta_J\wedge \eta_{J'})$. This product is $0$ if either $I\cap I'$ or $J\cap J'$ is non-empty. Otherwise, $(\lambda_I\otimes \eta_J)\star (\lambda_{I'}\otimes \eta_{J'})=\pm \lambda_{I\cup I'}\otimes \eta_{J\cup J'}$, where the sign depends on the index sets $I, I', J$ and $J'$. By Poincar\'e duality, we obtain a tropical Pontryagin product on cohomology, which we also denote by $\star$:
$$\begin{tikzcd}
{H^{p,q}(X)\times H^{p',q'}(X)} \arrow[d, "\sim"'] \arrow[rr, "\star"] &  & {H^{g-p-p',g-q-q'}(X)}                       \\
{H_{g-p,g-q}(X)\times H_{g-p', g-q'}(X)} \arrow[rr, "\star"]           &  & {H_{2g-p-p', 2g-q-q'}(X)} \arrow[u, "\sim"']
\end{tikzcd}$$
Under the identification $H^{\bullet, \bullet}(X)=\bigwedge^{\bullet}\Lambda^{\ast}\otimes \bigwedge^{\bullet}N^{\ast}$ and the description of the Poincar\'e duality map, we see that the cohomological Pontryagin product is given by: $$(\lambda^{\ast}_{I^o}\otimes \eta^{\ast}_{J^o})\star (\lambda^{\ast}_{I^{'o}}\otimes \eta^{\ast}_{J^{'o}})=\pm \lambda_{I^o\cap I^{'o}}\otimes \eta_{J^o\cap J^{'o}}\,,$$
where the sign is induced by the Pontryagin product on homology.

\begin{proposition}\label{prop:tropgenpoincare}
Let $[D]\in \operatorname{CaCl}(X)$ be an ample tropical Cartier divisor class on $X=N_\bR/\Lambda$ with $d=h^0(X,\sL(D))$. Let $c_{[D]}=[D]^{\cdot g-1}/(d(g-1)!)$, where $[D]^{\cdot g-1}$ is the $(g-1)$-fold tropical intersection product of $D$. Then, for $0\leq p\leq g$, we have the equality 
\begin{equation}\label{Eqn:genpoin}
\frac{[D]^{\cdot p}}{p!}=d\frac{c_{[D]}^{\star \ g-p}}{(g-p)!}\,,\end{equation}
in tropical homology.
\end{proposition}

\begin{proof}
  By the Poincar\'e duality isomorphism, it suffices to prove the Poincar\'e dual identity of \eqref{Eqn:genpoin} in tropical cohomology. Let $c_\sL\coloneqq c_1(\sL)^{\wedge g-1}/(d(g-1)!)$, where $\sL=\sL(D)\in \pic(X)$ is the line bundle corresponding to $[D]\in \operatorname{CaCl}(X)$. By definition of intersection product of line bundles, $c_\sL=\operatorname{cyc}(c_{[D]})$. Therefore, $c_{\sL}$ is Poincar\'e dual to $c_{[D]}$. Thus, the Poincar\'e dual of \eqref{Eqn:genpoin} is:
\begin{equation}\label{eqgenpoincoh}
    \frac{c_1(\sL)^{\wedge p}}{p!}=d\frac{c^{\star g-p}_{\sL}}{(g-p)!}\,,
\end{equation}
where $\star: H^{\bullet,\bullet}(X)\otimes H^{\bullet, \bullet}(X)\rightarrow H^{\bullet, \bullet}(X)$ in \eqref{eqgenpoincoh} is the cohomological Pontryagin product as described in \S\ref{subsubsec:cohompontryagin}. Let $\{\lambda_1,\dots, \lambda_g\}$ be a fixed basis of $\Lambda$ with corresponding dual basis $\{\lambda_1^{\ast},\dots, \lambda_g^{\ast}\}$ of $\Lambda^{\ast}$. Since $\sL$ is an ample line bundle, the first Chern class $c_1(\sL)$ defines an injective $\mathbb{Z}$-module homomorphism $\rho_{E}: \Lambda\rightarrow N^{\ast}$ between free $\bZ$-modules of equal rank $g$, as described in \S\ref{subsec:tropAV}. Thus, for each $1\leq i\leq g$, there exist $a_i\in \bZ$ such that $\{\eta^{\ast}_i\coloneqq \rho_E(\lambda_i)/a_i, \ 1\leq i\leq g\}$ form a $\bZ$-basis for $N^{\ast}$. We can write $c_1(\sL)=\sum_{i=1}^{g}a_i\lambda_i^{\ast}\otimes\eta^{\ast}_i$ as an element of $H^{1,1}(X)=\bigwedge \Lambda^{\ast}\otimes \bigwedge N^{\ast}$. Let $\phi_{\sL}:X\rightarrow \widehat{X}$ be the natural homomorphism induced by $\sL$. We can simplify the formula \eqref{eqnwedge4} for the pushforward on cohomology $\phi_{\sL\ast}: H^{g-p,g-p}(X)\rightarrow H^{g-p,g-p}(\widehat{X})$ as: 
\begin{equation}\label{eqnwedge5}
\phi_{\sL\ast}(\lambda^{\ast}_{I_{g-p}}\otimes \eta^{\ast}_{J_{g-p}})= a_{I^o_{g-p}}a_{J^o_{g-p}}\eta_{I_{g-p}}\otimes \lambda_{J_{g-p}}\,,
\end{equation}
where $a_K=\prod_{i\in K}a_i$ for any subset $K\subseteq [g]$. By \eqref{eqnwedge5} it follows that $\phi_{\sL\ast}(c_{\sL}^{\star g-p})=\phi_{\sL\ast}(c_{\sL})^{\star g-p}$. Using this and Corollary~\ref{maincor}, we have:
    \begin{equation}\label{eq:A}
    d\phi_{\sL\ast}\!\left(\frac{c_{\sL}^{\star \ g-p}}{(g-p)!}\right)\!= \! \frac{d}{(g-p)!}\phi_{\sL\ast}\!\left(\frac{c_1(\sL^{\cdot g-1})}{d(g-1)!}\right)^{\star g-p} \!\!\!\!\!\!\!= \frac{(-1)^{(g-1)(g-p)}d}{(g-p)!}F_{E_X}(c_1(\sL))^{\star g-p}\,.
    \end{equation}
    By Proposition~\ref{prop: fourierdesc}, we see that: $$F_{E_X}(c_1(\sL))=(-1)^{g-1}\sum_{i=1}^{g}a_i\lambda_{[g]\setminus\{i\}}\otimes \eta_{[g]\setminus\{i\}}\in H^{g-1,g-1}(\widehat{X})\,.$$ Then by the definition of the tropical Pontryagin product, we see that: \begin{equation}\label{eq:B}
    (-1)^{(g-1)(g-p)}\frac{F_{E_X}(c_1(\sL))^{\star g-p}}{(g-p)!}=\sum_{\substack{I\subseteq [g]\\ |I|=g-p}}a_I\lambda_{I^o}\otimes\eta_{I^o}=(-1)^pF_{E_X}\left(\frac{c_1(\sL^{\cdot g-p})}{(g-p)!}\right)\,.
    \end{equation}
    Thus, using \eqref{eq:A} and \eqref{eq:B}, we obtain:
    \[d\phi_{\sL\ast}\left(\frac{c_{\sL}^{\star \ g-p}}{(g-p)!}\right)=(-1)^pdF_{E_X}\left(\frac{c_1(\sL^{\cdot g-p})}{(g-p)!}\right)=\phi_{\sL\ast}\left(\frac{c_1(\sL^{\cdot p})}{p!}\right)\,.\]Since $\phi_{\sL\ast}: H^{p, p}(X)\rightarrow H^{p, p}(\widehat{X})$ is a bijection by \eqref{eqnwedge5}, this proves \eqref{eqgenpoincoh} and, by Poincar\'e duality, \eqref{Eqn:genpoin}.
\end{proof}

\begin{remark}\label{rem:nondegencase}
    While we prove Proposition~\ref{prop:tropgenpoincare} for ample line bundles, one can see that the same proof works for any nondegenerate tropical line bundle $\sL$ with $c_1(\sL)=E$, when $d$ in \eqref{Eqn:genpoin} is replaced by $\det E$.
\end{remark}

\section{Some consequences}\label{sec:applications}

In this final section we provide some consequences of Theorem~\ref{propgenpoin}.We provide some tropical analogues of classical results, including a tropical geometric Riemann--Roch theorem, as well as instances of special cases of our result which have previously appeared or conjectured in literature under the guise of tropical Poincar\'e formulas. 

\subsection{A tropical geometric Riemann--Roch theorem}\label{subsec:intersection}

In this section, we provide intersection number calculations for ample line bundles (or divisors) on tropical abelian varieties. We first note the case $p=g$ of Theorem~\ref{propgenpoin}, which provides the top self-intersection power of any ample divisor class on a tropical abelian variety.

\begin{corollary}\label{cor1}
    Let $[D]\in \operatorname{CaCl}(X)$ be an ample tropical Cartier divisor class on a tropical abelian variety $X$ of dimension $g$ with $d=h^0(X, \sL(D))$. Then we have the following equality: $$\frac{[D]^{\cdot g}}{g!}=d[1]\,,$$ in tropical homology. Here $[1]$ is the unit of the ring $(H_{\bullet,\bullet}(X), +, \star)$.
\end{corollary}

As a consequence of Corollary~\ref{cor1}, we obtain the following tropical analogue of geometric Riemann--Roch theorem for ample line bundles on abelian varieties (see \cite[\S16]{MumAV}, \cite[Theorem 3.6.3]{CAV}).

\begin{corollary}\label{cor:gRR}
    Let $\mathscr{L}\coloneqq \mathscr{L}(D)$ be an ample line bundle on a tropical abelian variety $X$ of dimension $g$ with $d=h^0(X, \mathscr{L})$. Then:
    \[d=\int_X\frac{c_1(\mathscr{L})^{\wedge g}}{g!}\,.\]
    Here $(D^g)\coloneqq \int_X c_1(\mathscr{L})^{\wedge g}$ is the $g$-fold intersection number of $D$, given by integration of tropical Dolbeault superforms.
\end{corollary}

\begin{proof}
Note that $c_1(\sL)^{\wedge g}\in H^{g,g}(X)$. By \cite[Theorem~1]{JSS}, one can identify the tropical cohomology group $H^{g,g}(X)$ with the Dolbeault cohomology group $H^{g,g}_{d''}(X)$ of superforms. Thus, we consider $c_1(\sL)^{\wedge g}$ as $(g,g)$-tropical superform on $X$. By the integration pairing of superforms \cite[Definition~4.11]{JSS}, we see that $\int_Xc_1(\sL)^{\wedge g}\in H^{0,0}_{d", c}(X)^{\ast}=H^{0,0}_{d''}(X)^{\ast}\cong H^{0,0}(X)^{\ast}$ (since $X$ is compact). By the explicit description of tropical homology and cohomology groups of a tropical abelian variety $X$, we see that $H^{0,0}(X)^{\ast}$ is naturally isomorphic to $H_{0,0}(X)$. Since $X$ is a tropical manifold, by \cite[Theorem~4.33]{JSS} the map $\int_X-: H^{g,g}_{d"}(X)\rightarrow H^{0,0}_{d"}(X)^{\ast}$ is the Poincar\'e duality isomorphism. However, the Poincar\'e duality isomorphism for tropical cohomology is given by $-\cap\operatorname{cyc}[X]: H^{g,g}(X)\rightarrow H_{0,0}(X)$ (see \cite[\S3F]{GSJ}). Thus, $\int_X c_1(\sL)^{\wedge g}= c_1(\sL)^{\wedge g}\cap \operatorname{cyc}[X]$ under the natural identification $H^{0,0}_{d"}(X)^{\ast}\cong H_{0,0}(X)$ for a real torus with integral structure. By the third equality of \cite{GSJ}*{Theorem~3.3}, it follows that $\operatorname{cyc}([D]^{\cdot g})=c_1(\sL)^{\wedge g}\cap \operatorname{cyc}[X]$. Thus, we obtain the following equality in $H_{0,0}(X)$:
    \begin{equation}\label{Eqn:geomRR}
    \frac{(D^{g})}{g!}\coloneqq\int_X\frac{c_1(\mathscr{L})^{\wedge g}}{g!}= \frac{[D]^{\cdot g}}{g!}=d[1]\,,\end{equation} 
    where the third equality follows from Corollary~\ref{cor1}. Since for a tropical abelian variety, the degree map (also denoted by $\int_X$) $H_{0,0}(X)\rightarrow \mathbb{Z}$ is an isomorphism mapping $[1]\in H_{0,0}(X)$ to $1\in\bZ$  (see \cite[\S3F]{GSJ}), we obtain the result by applying the degree map to \eqref{Eqn:geomRR}.
\end{proof}

\begin{remark}\label{rem:gRR}
     Corollary~\ref{cor:gRR}, in particular, proves that $(D^g)/g!$ is an integer. A combinatorial version (in terms of stable intersection numbers) of this fact is proven in \cite[Theorem~47]{KS}, answering a question of Cartwright.
\end{remark}

We note the following tropical analogue of \cite[Corollaire~3]{Beauv2}. We skip the proof as it is analogous to the classical case (see \cite[Corollary~16.5.8]{CAV}).

\begin{corollary}
    For $p, q\geq 0$ and an ample tropical Cartier divisor class $[D]$ on a tropical abelian variety $X$ of dimension $g$, we have the equality:
    \[ \frac{[D]^{\cdot p}}{p!}\star \frac{[D]^{\cdot q}}{q!}=d{2g-p-q\choose g-p}\frac{[D]^{\cdot p+q-g}}{(p+q-g)!}\,,\]
    in tropical homology.
\end{corollary}

\subsection{Applications to tropical Jacobians and Prym varieties}
In this final section, we obtain the tropical Poincar\'e formula \cite[Theorem A]{GSJ} modulo pure-dimensionality of the image of powers of tropical Abel---Jacobi map \cite[Theorem 8.3]{GST}.

The authors of \cite{RZ} prove the $d=1$ case of a tropical Poincar\'e--Prym formula \cite[Theorem 4.25]{RZ} and conjecture a similar formula \cite[Conjecture 4.25]{RZ} for higher values of $d$. Using the generalized Poincar\'e formula, we prove \cite[Conjecture 4.24]{RZ} in its full generality.
\subsubsection{Tropical Poincar\'e formula} \label{subsubsectropJac}
We briefly review the main objects involved in the tropical Poincar\'e formula and refer the reader to \cite{GSJ} for further details. Let $\Gamma$ be a compact tropical curve (metric graph) of genus $g$ and let $\Jac(\Gamma)$ be the associated tropical Jacobian, which is a tropical abelian variety via the canonical polarization induced by the tropical Riemann theta divisor $\Theta$. Let $\Phi^d_q: \Gamma^d\rightarrow\Jac(\Gamma)$ be the tropical Abel--Jacobi map with respect to a chosen base point $q\in\Gamma$. Then the image $\widetilde{W}_d\coloneqq\Phi^d_q(\Gamma^d)$ is a polyhedral subset of $\Jac(\Gamma)$ of pure dimension $d$ by \cite[Theorem 8.3]{GST}. Let $[\widetilde{W}_d]$ and $[\Theta]$ be the tropical fundamental classes supported on the polyhedral subsets $\widetilde{W}_d$ and $\Theta$ of $\Jac(\Gamma)$. The tropical Poincar\'e formula \cite[Theorem A]{GSJ} relates these classes in tropical homology as follows:
\begin{equation}\label{Eqn:Poinformula}
    [\widetilde{W}_d]=\frac{\Phi^d_{q,*}[\Gamma^d]}{d!}=\frac{[\Theta]^{\cdot g-d}}{(g-d)!}\,.
\end{equation}
The first equality in \eqref{Eqn:Poinformula} above is a consequence of pure-dimensionality of $\widetilde{W}_d$ (see \cite[Proposition~8.3]{GSJ}). We show that the second equality in \eqref{Eqn:Poinformula} can be obtained from the generalized Poincar\'e formula Theorem~\ref{propgenpoin}. To see this, apply Theorem~\ref{propgenpoin} to the (ample) theta divisor class $[\Theta]\in \operatorname{CaCl}(\Jac(\Gamma))$ to obtain: 
\[\frac{c_{[\Theta]}^{\star d}}{d!}=\frac{[\Theta]^{\cdot g-d}}{(g-d)!}\,,\]
for any $0\leq d\leq g$. By definition, we have $c_{[\Theta]}=[\Theta]^{\cdot g-1}/(g-1)!$. By choosing appropriate bases of $H^{\ast, \ast}(\Jac(\Gamma))$ (see \cite{GSJ}*{\S9A}), one can see that $c_{[\Theta]}=\Phi_{q,\star}[\Gamma]=[\widetilde{W}_1]$ from \cite{GSJ}*{Proposition~9.2} and \cite{GSJ}*{Lemma~9.5}. Thus, it suffices to prove that $c_{[\Theta]}^{\star d}=(\Phi_{q,\star}[\Gamma])^{\star d}=\Phi^d_{q,\star}[\Gamma^d]$ for any $2\leq d\leq g$. But this follows from the following commutative diagram:
\begin{equation}\label{Eqn:Commdiag1}
\begin{tikzcd}
	{\prod_{i=1}^{d}H_{\bullet,\bullet}(\Gamma)} && {\prod_{i=1}^{d}H_{\bullet,\bullet}(\Jac(\Gamma))} \\
	\\
	{\bigotimes_{i=1}^{d}H_{\bullet,\bullet}(\Gamma)=H_{\bullet,\bullet}(\Gamma^d)} && {H_{\bullet,\bullet}(\Jac(\Gamma))}
	\arrow["{\prod_{i=1}^{d}\Phi_{q,\ast}}", from=1-1, to=1-3]
	\arrow[from=1-1, to=3-1]
	\arrow["\star", from=1-3, to=3-3]
	\arrow["{\Phi^d_{q,\ast}}"', from=3-1, to=3-3]
\end{tikzcd}
\end{equation}
The equality in the bottom left corner of the above commutative diagram is by the K\"unneth isomorphism (\cite[Theorem B]{GSH}). The vertical right arrow in Diagram~\eqref{Eqn:Commdiag1} is the $d$-fold Pontryagin product sending a $d$-tuple to the Pontryagin product of the elements of the $d$-tuple. The above diagram commutes since the composition of the top and right maps is multilinear, so it factors through the tensor product $\mathbb{Z}$-algebra on the bottom left. Using the fact that the left vertical map sends the $d$-tuple $(\operatorname{cyc}[\Gamma], \operatorname{cyc}[\Gamma],\dots,\operatorname{cyc}[\Gamma])$ to $\operatorname{cyc}[\Gamma^d]$ and the identity \eqref{Eqn:Poinformula} for $d=1$, it follows that, for any $1\leq d\leq g$, we have $c_{[\Theta]}^{\star d}=(\Phi_{q,\star}[\Gamma])^{\star d}=\Phi^d_{q,\star}[\Gamma^d]$. This yields the second equality of \eqref{Eqn:Poinformula} as a particular instance of Theorem~\eqref{propgenpoin}. 

\subsubsection{Tropical Poincar\'e--Prym formula}\label{subsubsectropprym} In this section, we prove \cite[Conjecture 4.24]{RZ} and establish the tropical Poincar\'e--Prym formula. We briefly recall the relevant background and refer the reader to \cite{RZ} for further details.

Let $\pi: \widetilde{\Gamma} \rightarrow \Gamma$ be a harmonic double cover of compact tropical curves. To $\pi$ one can associate a tropical Prym variety, analogous to the classical case, which has been defined and studied in \cite{JL}, \cite{LU}, \cite{LZ} and \cite{RZ23}. Following \cite{RZ}, we call this the divisorial Prym variety of the double cover and denote it by $\operatorname{Prym}_d(\widetilde{\Gamma}/\Gamma)$. In \cite[\S4.4]{RZ}, the authors introduce a variant of the tropical Prym variety, called the \emph{continuous} tropical Prym variety $\operatorname{Prym}_c(\widetilde{\Gamma}/\Gamma)$ associated to the double cover. The authors also demonstrate the relationship between $\operatorname{Prym}_d(\widetilde{\Gamma}/\Gamma)$ and $\operatorname{Prym}_c(\widetilde{\Gamma}/\Gamma)$ and propose evidence in support of $\operatorname{Prym}_c(\widetilde{\Gamma}/\Gamma)$ being a more natural and well-behaved object than $\operatorname{Prym}_d(\widetilde{\Gamma}/\Gamma)$.

Let $\iota: \widetilde{\Gamma}\rightarrow \widetilde{\Gamma}$ be the natural involution associated to the double cover $\pi:\widetilde{\Gamma}\rightarrow\Gamma$. Then, by \cite[Proposition~4.18(2)]{RZ}, the induced map $\operatorname{Id}-\iota:\Jac(\widetilde{\Gamma})\rightarrow\Jac(\widetilde{\Gamma})$ factorizes as:
\begin{equation}\label{Eqn:factorize}
\Jac(\widetilde{\Gamma})\xrightarrow{\epsilon}\operatorname{Prym}_c(\widetilde{\Gamma}/\Gamma)\xrightarrow{\gamma}\operatorname{Prym_d(\widetilde{\Gamma}/\Gamma})\xhookrightarrow{i}\Jac(\widetilde{\Gamma}/\Gamma)\,.
\end{equation}

By \cite[Proposition~4.18]{RZ}, there is a natural principal polarization $\zeta_c$ on $\operatorname{Prym_c(\widetilde{\Gamma}/\Gamma)}$ such that $\zeta\coloneqq2\zeta_c$ is the polarization induced on $\operatorname{Prym_c(\widetilde{\Gamma}/\Gamma)}$ from the canonical principal polarization on $\Jac(\widetilde{\Gamma})$ via $i\circ \gamma$. Analogous to tropical Abel--Jacobi maps, there also exist tropical $d$-fold Abel--Prym maps $\Psi^d_q:\widetilde{\Gamma}^d\rightarrow \operatorname{Prym}_c(\widetilde{\Gamma}/\Gamma)$ with respect to base point $q\in\widetilde{\Gamma}$ for all $d\geq 1$ (see \cite[Definition 4.21]{RZ}). Our final result is the following corollary of Theorem~\ref{propgenpoin}, which is a tropical analogue of algebraic Poincar\'e formula for Prym varieties. This resolves Conjecture~4.24 in \cite{RZ}.

\begin{corollary}[The tropical Poincar\'e--Prym formula]\label{conPP}
    Let $\widetilde{\Gamma} \rightarrow \Gamma$ be a double cover of tropical curves, let $q \in \widetilde{\Gamma}$ be a basepoint and let $g_0=\operatorname{dim} \operatorname{Prym}_c(\widetilde{\Gamma} / \Gamma)\coloneqq g(\widetilde{\Gamma})-g(\Gamma)$. Then:
$$
\frac{\Psi^d_{q,*}\operatorname{cyc}[\widetilde{\Gamma^d}]}{d!}=\frac{2^{d}}{\left(g_0-d\right) !}[\zeta_c]^{g_0-d} \in H_{d, d}(\operatorname{Prym}(\widetilde{\Gamma} / \Gamma))\,,
$$
for $1\leq d\leq g_0$, where $[\zeta_c]\in H_{g_0-1,g_0-1}(\operatorname{Prym}(\widetilde{\Gamma}/\Gamma))$ is the class of the principal polarization of $\operatorname{Prym}_c(\widetilde{\Gamma} / \Gamma)$.
\end{corollary} 

\begin{proof}
The Conjecture is true for $d=1$ by \cite[Theorem 4.25]{RZ}, that is: 
\begin{equation}\label{Eqn:d=1PoinPrym}
\Psi_{q \star}\operatorname{cyc}[\widetilde{\Gamma}]=\frac{2}{(g_0-1)!}[\zeta_c]^{\cdot g_0-1}=:2c_{\zeta_c}\in H_{1,1}(\operatorname{Prym_c(\widetilde{\Gamma}/\Gamma))}.
\end{equation}
Using \eqref{Eqn:d=1PoinPrym} and Theorem~\ref{propgenpoin} we obtain, for all $1\leq d\leq g_0$:
\begin{equation}\label{Eqn:PoinPrym2}
    \frac{(\Psi_{q\star}\operatorname{cyc}[\widetilde{\Gamma}])^{\star d}}{d!}=\frac{(2c_{\zeta_c})^{\star d}}{d!}=2^d\frac{c^{\star d}_{\zeta_c}}{d!}=\frac{2^d}{(g_0-d)!}[\zeta_c]^{\cdot g_0-d} \in H_{d,d}(\operatorname{Prym}_c(\widetilde{\Gamma}/\Gamma))\,.
\end{equation}
Analogous to \eqref{Eqn:Commdiag1}, we have the following commutative diagram:
\begin{equation}\label{Eqn:Commdiag2}
\begin{tikzcd}
	{\prod_{i=1}^{d}H_{\bullet,\bullet}(\widetilde{\Gamma})} && {\prod_{i=1}^{d}H_{\bullet,\bullet}(\operatorname{Prym}_c(\widetilde{\Gamma}/\Gamma))} \\
	\\
	{\bigotimes_{i=1}^{d}H_{\bullet,\bullet}(\widetilde{\Gamma})=H_{\bullet,\bullet}(\widetilde{\Gamma}^d)} && {H_{\bullet,\bullet}(\operatorname{Prym}_c(\widetilde{\Gamma}/\Gamma))}
	\arrow["{\prod_{i=1}^{d}\Psi_{q,\ast}}", from=1-1, to=1-3]
	\arrow[from=1-1, to=3-1]
	\arrow["\star", from=1-3, to=3-3]
	\arrow["{\Psi^d_{q,\ast}}"', from=3-1, to=3-3]
\end{tikzcd}
\end{equation}
 As before, we observe that the left map sends the $d$-tuple $(\operatorname{cyc}[\widetilde{\Gamma}], \operatorname{cyc}[\widetilde{\Gamma}],\dots,\operatorname{cyc}[\widetilde{\Gamma}])$ to $\operatorname{cyc}[\widetilde{\Gamma}^d]$, whereby it follows that for any $1\leq d\leq g_0$, we have $(\Psi_{q,*}\operatorname{cyc}[\widetilde{\Gamma}])^{\star d}=\Psi^d_{q, *}\operatorname{cyc}[\widetilde{\Gamma}^d]$. Using this in \eqref{Eqn:PoinPrym2}, the equality follows.
\end{proof}

\vspace{3mm}

\end{document}